\documentclass{siamltex1213}
\usepackage[latin1]{inputenc}
\usepackage[english]{babel}
\usepackage{amsmath}
\usepackage[x11names,usenames,dvipsnames,svgnames,table,rgb]{xcolor}
\usepackage{amsfonts}
\usepackage{amssymb}
\usepackage{graphicx}
\usepackage{ulem}
\usepackage{tikz}
\usetikzlibrary{arrows}

\newcommand{\PP}{\mathbb{P}}

\newcommand{\EE}{\mathbb{E}}
\newcommand{\RR}{\mathbb{R}}
\newcommand{\NN}{\mathbb{N}}

\newcommand{\spt}{\mathrm{spt}}
\newcommand{\limk}{\underset{k\to\infty}{\longrightarrow}}

\newtheorem{assumption}{Assumption}
\newtheorem{remark}{Remark}
\newtheorem{example}{Example}
\newcommand{\add}[1]{#1}
\newcommand{\del}[1]{}

\title{\bf Linear conic optimization\\ for inverse optimal control\footnote{This work was partly funded by the ERC Advanced Grant Taming.}}

\begin{document}

\author{Edouard Pauwels$^{1}$, Didier Henrion$^{2,3,4}$, Jean-Bernard Lasserre$^{2,3}$}
\footnotetext[1]{IRIT-IMT; Universit\'e Paul Sabatier; 118 route de Narbonne 31062, Toulouse Cedex 9, France.}
\footnotetext[2]{CNRS; LAAS; 7 avenue du colonel Roche, F-31400 Toulouse; France.}
\footnotetext[3]{Universit\'e de Toulouse;  LAAS, F-31400 Toulouse, France.}
\footnotetext[4]{Faculty of Electrical Engineering, Czech Technical University in Prague,
Technick\'a 2, CZ-16626 Prague, Czech Republic}

\date{Draft of \today}

\maketitle

\begin{abstract}
We address the inverse problem of Lagrangian identification based on trajectories in the context of nonlinear optimal control. We propose a general formulation of the inverse problem based on occupation measures and complementarity in linear programming. The use of occupation measures in this context offers several advantages from the theoretical, numerical and statistical points of view. We propose an approximation procedure for which strong theoretical guarantees are available. Finally, the relevance of the method is illustrated on academic examples.
\end{abstract}


\section{Introduction}

In the context of nonlinear optimal control, we are interested in the inverse problem of Lagrangian identification from given trajectories. This identification should be carried out such that solving the direct optimal control problem with the identified Lagrangian would allow to recover the given trajectories.

Inverse problems of calculus of variations are old topics that have attracted a renewal of interest in the context of optimal control, especially in humanoid robotics \cite{arechavaleta2008optimality}. Relevant aspects of the problem are not well understood and many issues still need to be addressed  to propose a tool that could be used in experimental settings. The work presented here constitutes a step in this direction. A preliminary conference version \cite{pauwels2014inverse} originally introduced our optimization framework as a tool to solve the inverse problem numerically. The current paper extends this work in many ways.
In particular, by using the (quite general) concept of {\it occupation measures} we can propose a broad definition of inverse optimality and we also rigorously
justify most of the approximations behind the numerical results reported in \cite{pauwels2014inverse}. Many aspects of this work parallel the results of \cite{lasserre2008nonlinear} about {\it direct} optimal control with polynomial data.

\subsection{\add{Motivation}}
\add{
The principle of optimality (or stationarity) is very important as a conceptual tool to describe laws of phenomenon are observed in nature (\textit{e.g.} Fermat's principle in optics, Lagrangian dynamics in mechanics). Beyond physics, similar tools and arguments are used to describe and model the behaviour of living systems in biology \cite{rosen1967optimality} or decision making agents in economics \cite{kamien1991dynamic}. Of more important interest to us is the application of the optimality principle to model the motion of living organisms \cite{todorov2004optimality}. In our technological context, this constitutes a hot topic. Promising expectations for these types of model include:
\begin{itemize}
	\item The conceptual understanding of general laws that govern decision taking processes related to living organism motion, including human motion \cite{arechavaleta2008optimality}.
	\item The {ability} to use these general laws to reproduce and synthetise motion behaviours for new tasks with unknown space configuration.
\end{itemize}
In this context, the principle of optimality only constitutes one possible conceptual tool to understand motion. There is a debate regarding its validity \cite{friston2011what} or its direct applicability in robotics applications \cite{laumond2014optimality}. These illustrate the fact that this idea constitutes an active subject of research, with a strong connextion with  applications. 

In many situations however, the cost related to the motion of a system is unknown or does not correspond to direct intuition. In these cases, as clearly emphasized in \cite{todorov2004optimality}: ``It would be very useful to have a general data analysis procedure that infers the cost function given experimental data and a biomechanical model''. {Our contribution is to investigate} the mathematical meaning of ``inferring cost function from data'' and we propose a numerical method to address problems of this type based on inverse optimality. {We emphasize} that this paper is ``only'' concerned with this question. {In particular we do not address the issue of interpreting} the inferred cost function or solving direct problems {for} new unseen conditions. We solely focus on the task of inferring a cost function from data. This constitutes a nontrivial shift in term of point of view compared to usual questions arising when dealing with direct optimal control problems. We hope to convince the reader that there are crucial differences between inverse and direct optimal control and that it is worth investigating the former within an appropriate context with somewhat different questions in mind.

The backbone of the proposed approach and its relation with the direct problem of optimal control is presented in Figure \ref{fig:directIllustr}. It is important to understand the symmetric role of the Lagrangian and the occupation measure representing the input trajectories. As a matter of fact, since the input of the inverse problem is a set of trajectories {(supposedly optimal for a certain Lagrangian)}, many aspects of the existence of minimizers 
{that are crucial in direct optimal control, are not relevant} for inverse problems since the {``optimal"} trajectories are given. For example, {there is no need to recompute} optimal trajectories for direct problems with initial conditions {already considered in the input data} since by inverse optimality, the {input trajectories are optimal} with respect to the identified Lagrangian.
}
\add{
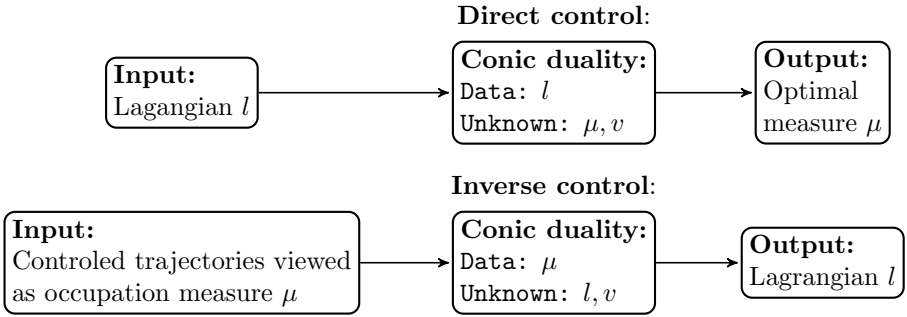
\begin{figure}
	\begin{center}
		\begin{tikzpicture}[
		  every matrix/.style={ampersand replacement=\&,column sep=1cm,row sep=.1cm},
		  oneBox/.style={draw,thick,rounded corners,inner sep=.1cm,align=left},
		  twoBox/.style={align=left},
		  to/.style={->,>=stealth',shorten >=1pt,semithick,font=\sffamily\footnotesize}]
		
		  \matrix{
				\node[twoBox] (box01) {\underline{\textbf{System description}: $f, X, U$}};\&\\
				\&\node[twoBox] (box01) {\flushleft{\textbf{Direct control}:}};\\
				\node[oneBox] (box1) {\textbf{Input:}\\Lagangian $l$};
				\& \node[oneBox] (box2) {\textbf{Conic duality:}\\
																			\texttt{Data:} $l$\\
																			\texttt{Unknown:} $\mu, v$
		
							};
				\& \node[oneBox] (box3) {\textbf{Output:}\\Optimal\\measure $\mu$};\\
				\&\\
				\&\\
				\&\node[twoBox] (box01) {\textbf{Inverse control}:};\\
				\node[oneBox] (box21) {\textbf{Input:}\\Controled trajectories viewed\\as occupation measure $\mu$};
				\& \node[oneBox] (box22) {\textbf{Conic duality:}\\
																			\texttt{Data:} $\mu$\\
																			\texttt{Unknown:} $l, v$
		
							};
					\& \node[oneBox] (box23) {\textbf{Output:}\\Lagrangian $l$};\\
		  };
		
		  \draw[to] (box1) -- (box2);
		  \draw[to] (box2) -- (box3);
		  \draw[to] (box21) -- (box22);
		  \draw[to] (box22) -- (box23);
		\end{tikzpicture}
	\end{center}
	\caption{\label{fig:directIllustr} \add{Direct optimal and inverse optimal control flow chart. System description is given by the dynamics $f$, the state constraint set $X$ and control constraint set $U$ which are all assumed to be fixed. We emphasize that the Lagrangian and the occupation measure have symmetric roles for the direct and inverse problems. We also note that the output of the inverse problem is a Lagrangian. {Solving} the direct optimal control problem for new initial conditions is an important question but remains secondary regarding inverse optimality which is the focus of this work.}}
\end{figure}
}
\subsection{{Context}}
Since its introduction by Kalman \cite{kalman1964linear}, the inverse problem of optimal control has been studied in linear settings \cite{anderson1971linear,jameson1973inverse,fujii1984complete,nori2004linear} leading to many nonlinear variations \cite{thau1967inverse, moylan1973nonlinear,casti1980general,freeman1996inverse}. In these works the input data of the problem is a characteristic of a class of trajectories often given in the form of a control law. This contrasts with the setting we propose to study, for which the input is a set of trajectories which could come from physical experiments. This motivates the work of \cite{chittaro2013inverse} and \cite{ajami2013humans} about well-posedness of the inverse problem, both in the context of unicycle dynamics in robotics and strictly convex positive Lagrangians.

On the other hand,  to treat the inverse problem  several authors have proposed numerical methods
based on the ability to solve the direct problem \cite{mombaur2010human}, also in the context of Markov decision process \cite{abbeel2004apprenticeship,ratliff2006maximum} or based on a discretized version of the direct problem \cite{puydupin2012convex,keshavarz2011imputing}.

Our approach is different and based on occupation measures, an abstract and quite general tool to handle trajectories (\add{and their weak limits}) of feasible solutions of classical control problems. Formulating the (direct) control problem on appropriate spaces of measures amounts to relaxing the original problem. In most applications, both relaxed and original problems have same optimal value \cite{vinter1993convex,vinter1978equivalence,gaitsgory2009linear}. However the relaxed formulation has the crucial advantage that compactness holds in a certain weak sense: As a matter of fact, many optimization problems over appropriate spaces of measures attain their optimum, whereas most optimization problems over smaller functional spaces (e.g. continuous functions, or Lebesgue integrable functions) typically have {\it no} optimal solution. 
At last but not least, for control problems with polynomial data, the relaxed problem can be formulated as an optimization problem
on {\it moments} of occupation measures. By  combining this  with relatively recent 
advances in real algebraic geometry \cite{putinar1993positive} and in numerical optimization \cite{lasserre2010} one may thus
provide a systematic numerical scheme 
to approximate effectively relaxed solutions of optimal control problems \cite{lasserre2008nonlinear}.

\subsection{Contribution}
We choose the setting of free terminal time optimal control which is consistent with many physical experiments that one can think of.
But the same approach with {\it ad hoc} modifications is also  valid in the fixed terminal time setting.

$\bullet$ In our opinion,
occupation measures are the perfect abstract tool to formally express the fact that we consider a (possibly uncountably infinite) superposition of trajectories as input data of the inverse control problem. We then propose a general formulation of the inverse problem based on occupation measures and complementarity in linear programming. A relaxation of the well known Hamilton-Jacobi-Bellman (HJB) sufficient optimality condition appears in our formulation as for the usual direct optimal control problem \cite{hernandez1996linear}. This formulation is shown to be consistent with what is commonly expected regarding inverse optimality. 

It is worth noting that when using the HJB optimality conditions, the situation is completely symmetric for the direct and inverse control problems. In both cases the HJB optimality conditions are used to {\it certify the global optimality of trajectories}. But in the former 
the Lagrangian is known and HJB provide conditions on the optimal state-control trajectories (to be determined) whereas in the latter
the ``optimal" state-control trajectories are known and HJB provide conditions on the Lagrangian (to be determined) for the given trajectories to be optimal.
(In both cases the optimal value function is considered as an auxiliary ``variable".)

$\bullet$ Furthermore, this framework allows to further characterize the space of solutions associated with a given inverse optimal control problem. 
This viewpoint is different from what has been proposed in previous (theoretical and numerical) contributions to  this problem \cite{mombaur2010human,puydupin2012convex,chittaro2013inverse,ajami2013humans} which, implicitly or explicitly, involve strong
(and, in our opinion, overly restrictive) constraints on the class of functions in which the candidate Lagrangians are searched.

\add{$\bullet$ {The weak formulation of 
direct optimal control problems via occupation measures is elegant and powerful but also involves difficult technical questions regarding} potential gaps between classical and generalized control problems. Using inverse optimality, we justify \textit{a posteriori} that this discussion can be partially mitigated for the inverse problem. This striking difference between direct and inverse problems is {due to the symmetric roles of the Lagrangian and occupation measure} and the fact that the occupation measure is given and fixed for the inverse problem.} 

$\bullet$ Remarkably, despite the abstract setting of occupation measures,
the proposed formulation is amenable to explicit numerical approximations via a hierarchy of semi-definite programs\footnote{
A semi-definite program is a finite-dimensional linear optimization problem over the cone of non-negative quadratic forms for which
powerful primal-dual interior-point algorithms are available \cite{vandenberghe1996semidefinite}.}.
Indeed in the context of polynomial dynamics and semi-algebraic constraints, both the optimal value function and Lagrangian
used in the (relaxed) HJB optimality conditions can be approximated with polynomials. We show that such a reinforcement is coherent in the sense that
no polynomial solution to the inverse problem is lost.

$\bullet$ Finally, in usual experimental settings one does not have access to complete  trajectories. Instead one is rather given
finitely many data points sampled from trajectories. But
results from probability applied to our occupation measures allow to formalize the fact that we only work with ``samples''. 
In addition, in this framework one may use empirical processes and statistical learning theory \cite{vapnik1999overview,bousquet2004introduction} to provide bounds on the error made when working with samples instead of original trajectories.

\paragraph{Organization of the paper.} In Section \ref{sec:prelim} we provide the context and background on optimal control and occupation measures. In Section \ref{sec:iocp}, we present our characterization of solutions to the inverse optimal control problem and illustrate how it allows to further discuss about
the set of solutions and links with the direct optimal control problem. Numerical approximations via polynomials and statistical approximations via
finite samples are provided and discussed  in Section \ref{sec:practical}. The resulting numerical scheme
(with proven strong theoretical guarantees) can be implemented with off-the-shelf software on a standard computer.
Finally, Section \ref{sec:numerics} describes numerical results on academic examples.

\section{Preliminaries}
\label{sec:prelim}

\subsection{Notations}

If $A$ is a compact subset of a finite-dimensional Euclidean space, let $\mathcal{C}(A)$ resp.  $\mathcal{C}^1(A)$ denote the set of continuous resp. continuously differentiable functions from $A$ to $\RR$. Let $\mathcal{M}(A)$ denote the space of Borel measures on $A$, the topological dual of $\mathcal{C}(A)$ with duality bracket denoted by $\left\langle.,.\right\rangle$, i.e. $\left\langle \mu,f \right\rangle = \int_A f(x)d\mu(x)$ is the integration on $A$ of a function $f \in \mathcal{C}(A)$ with respect to a measure $\mu \in \mathcal{M}(A)$. Let $\mathcal{M}_+(A)$ resp. $\mathcal{C}_+(A)$ denote the cone of non-negative Borel measures resp. non-negative continuous functions on $A$. The support of a measure $\mu \in \mathcal{M}_+(A)$ is denoted by $\spt\:\mu$.
An element $\mu \in \mathcal{M}_+(A)$ such that $\left\langle \mu,1 \right\rangle = 1$
is called a probability measure. Let $\delta_x$ denote the Dirac measure concentrated on $x$ and let $I(e)$ denote the indicator function of an event $e$, equal to $1$ if $e$ is true, and $0$ otherwise.

Let $X \subseteq \RR^{d_X}$ denote the state space and $U \subseteq \RR^{d_U}$ denote the control space which are supposed to be compact subsets of Euclidean spaces. System dynamics are given by a continuously differentiable vector field $f \in \mathcal{C}^1(X \times U)^{d_X}$. Terminal state constraints are modeled by a set $X_T \subset X$ which is also given.
Let $B_n$ denote the unit ball of the Euclidean norm in $\RR^n$, 
and let $\partial S$ denote the boundary of set $S$ in the Euclidean space.   Let
$\RR[z]$ denote the set of multivariate polynomials with real coefficients with variables $z$
and let $\RR_k[z]$ denote the set of such polynomials with degree at most $k$. For a polynomial $p \in \RR_k[z]$,
we denote by $\|p\|_1$ the sum of the absolute values of the coefficients of $p$ when expanded in the monomial basis.

\subsection{Context: free terminal time optimal control}
We consider direct optimal control problems of the form:
\begin{equation}\label{eq:pdirect0}
	\tag{ocp$_0$}
\begin{array}{ll@{\;}l@{}}
	v_0(z) :=&  \displaystyle\inf_{u, T} & \displaystyle\int_0^T l_0(x(t), u(t)) dt \\
	& \mathrm{s.t.} & \dot{x}(t) = f(x(t), u(t)), \\
           && x(t) \in X,\,u(t) \in U,\, t \in [0,T],\\
	&&x(0) = z, \,x(T) \in X_T,\\
	&&T \in [0,T_M]
\end{array}
\end{equation}
with Lagrangian $l_0 \in \mathcal{C}(X \times U)$ and free final time $T$ with a given upper bound $T_M$ which ensures that
the value function $v_0$ is bounded below. Dynamics $f$ are given, as well as the sets $X$, $U$ and $X_T \subset X$. We assume that
a set $X_0 \subset X$ is given such that the following assumption is satisfied:
\begin{assumption}
	\label{ass:feasibility}
	For all initial conditions $z \in X_0$, problem (\ref{eq:pdirect0}) is feasible.
\end{assumption}

\subsection{Occupation measures}
\label{sec:OM}
\add{In this section we describe how to construct an occupation measure from a feasible trajectory of (\ref{eq:pdirect0}) and then from a set of such trajectories. The content of this section was already described in the litterature (see for example \cite{lasserre2008nonlinear,henrion2014convex, henrion2014optim}) {and} we include these notions here for completeness. Let $z \in X_0$ be an initial point. We use Assumption \ref{ass:feasibility} to fix a trajectory starting from $z$. That is, a terminal time $T_z \in \RR_+$, a measurable control $u_z \colon [0, T_z] \to U$ and an absolutely continuous trajectory $x_z \colon [0, T_z] \to X$ such that
\begin{equation}\label{eq:supperpos1}
	\begin{array}{c}
\dot{x}_z(t)  = f(x_z(t), u_z(t)),\\
x_z(0) = z, \:\: x_z(T_z) \in X_T.
	\end{array}
\end{equation}
 
The occupation measure of the corresponding trajectory is denoted by $\mu_z$ and is defined by
\begin{equation}\label{eq:supperpos1.1}
		\mu_z(A \times B)  := \int_0^{T_z} I(x_z(t) \in A, \;u_z(t) \in B)dt
\end{equation}
for every Borel sets $A \subset X$ and  $B \subset U$. We now turn to the construction of \textit{occupation measure} and \textit{terminal measure} of a set of trajectories by taking a measurable combination of occupation measures of single trajectories}. Consider a probability measure $\mu_0 \in \mathcal{M}_+(X_0)$ \add{and an upper bound on terminal time $T_M$}. Thanks to Assumption \ref{ass:feasibility}, for each $z \in  \spt\:\mu_0$, we fix a terminal time $T_z \in [0,T_M]$, a measurable control $u_z \colon [0, T_z] \to U$ and an absolutely continuous trajectory $x_z \colon [0, T_z] \to X$ such that (\ref{eq:supperpos1}) holds. 
\add{That is, for each $z \in  \spt\:\mu_0$, we have an occupation measure $\mu_z$ as described in (\ref{eq:supperpos1.1}).} The occupation measure $\mu \in \mathcal{M}_+(X \times U)$ and terminal measure $\mu_T \in \mathcal{M}_+(X_T)$ of the set of trajectories $\{x_z(t)\}_{z \in \spt\:\mu_0, t \in [0,T_z]}$ are then defined \add{by tacking a convex combination of each $\mu_z$ according to $\mu_0$} (see also \cite[Chapter 5]{henrion2014optim} and \cite[Section 3]{henrion2014convex}). We obtain the following definition:
\begin{equation}\label{eq:supperpos2}
	\begin{array}{ccl}
		\mu(A \times B) & :=& \displaystyle \int_{X_0} \mu_z(A, B)  \mu_0(dz),\\
		&=& \displaystyle \int_{X_0} \left(\int_0^{T_z} I(x_z(t) \in A, \;u_z(t) \in B)dt \right) \mu_0(dz),\\
		\mu_T(A) &:=& \displaystyle \int_{X_0} I(x_z(T_z) \in A)\; \mu_0(dz),
	\end{array}
\end{equation}
for every Borel sets $A \subset X$ and  $B \subset U$. With the previous definition, 
\[
	\left\langle \mu, l \right\rangle = \int_{X_0} \left(\int_0^{T_z} l(x_z(t), u_z(t))dt \right) \mu_0(dz),\quad\forall l \in \mathcal{C}(X \times U).
\]
In particular
\[
\mu(X \times U) = \left\langle \mu,1\right\rangle = \int_{X_0} T_z\; \mu_0(dz).
\]
Furthermore for every $v \in \mathcal{C}^1(X)$, 
\begin{equation}\label{eq:liouville1}
	\begin{array}{ccl}
		\left\langle\mu, \mathrm{grad}\:v \cdot f\right\rangle &=& \displaystyle \int_{X_0} \left(\int_0^{T_z} \mathrm{grad}\:v(x_z(t)) \cdot f(x_z(t), u_z(t))\;dt\right) \mu_0(dz)\\
		&=& \displaystyle \int_{X_0} \left(v(x_z(T_z)) - v(x_z(0)) \right)\mu_0(dz) \\
		&=& \left\langle \mu_T, v \right\rangle - \left\langle\mu_0, v \right\rangle,
	\end{array}
\end{equation}
where ``$\mathrm{grad}$" denotes the gradient vector of first order derivatives of $v$, and the ``dot"
denotes the inner product between vectors.
Equation (\ref{eq:liouville1}) is known as Liouville's equation and is also written as
\begin{equation}\label{eq:liouville2}
	\mathrm{div}f \mu + \mu_T = \mu_0,
\end{equation}
	where the divergence is to be interpreted in the weak sense and a change of sign comes from integration by part. As we have seen, occupation and terminal measures as defined in (\ref{eq:supperpos2}) satisfy the Liouville equation (\ref{eq:liouville2}). \add{This motivate the following broader definition. 
		\begin{definition}\label{def:occMeasure}
		A general occupation measure is a measure that satisfies Liouville's equation (\ref{eq:liouville2}), for some terminal measure $\mu_T \in \mathcal{M}_+(X_T)$, in the weak sense described in (\ref{eq:liouville1}).
	\end{definition} 

	We have seen in this section how to construct an occupation measure from a set of feasible trajectories of (\ref{eq:pdirect}). However the set of all occupation measures is in general much bigger than the set of measures arising in this way.}
\subsection{\add{Input of the inverse optimal control problem}}
\add{For inverse optimal control, we suppose that the trajectories are given. Moreover, the Liouville equation and positivity constraints are sufficient to develop all the aspects of our analysis of inverse optimality. 
\begin{center}
	{\it Therefore, independently of how it is constructed, the input data of our inverse control problem is a general occupation measure as given by Definition \ref{def:occMeasure}. }
\end{center}
This restriction is made without loss of generality regarding classical trajectories because, from the construction in (\ref{eq:supperpos2}), we consider an input set that contains all of them. All the results will in particular apply to situations when the occupation measure is a superposition of classical trajectories as described in (\ref{eq:supperpos2}). The results will also hold if this is not the case and the input measure involves generalized control. Finally and most importantly, this construction allows to {formally} treat cases for which we are given a possibly uncountably infinite number of trajectories as input data and is therefore much more general than considering one or a few classical trajectories.}

\subsection{Direct optimal control}\label{sec:OCP}
Using the formalism of occupation measures, given a continuous Lagrangian $l$, an initial measure $\mu_0$ and a maximal terminal time $T_M$, we consider direct optimal control problems of the form
\begin{equation}\label{eq:pdirect}
	\tag{ocp}
\begin{array}{ll@{\;}l@{}}
	p^*(\mu_0) :=& \displaystyle\inf_{\mu, \mu_T} & \left\langle \mu, l\right\rangle \\
	& \mathrm{s.t.} & \mathrm{div}f \mu + \mu_T = \mu_0,\\
           && \left\langle \mu,1 \right\rangle \leq T_M,\\
					 &&\mu \in \mathcal{M}_+(X \times U),\\
					 &&\mu_T \in \mathcal{M}_+(X_T).
\end{array}
\end{equation}

\begin{definition}[OCP]
$\mathrm{OCP}(l, \mu_0, T_M)$ is the set of measures $(\mu, \mu_T)$ solving problem (\ref{eq:pdirect}). 
\end{definition} 

Note that by Lemma \ref{lem:duality} and Assumption \ref{ass:feasibility}, set $\mathrm{OCP}(l, \mu_0, T)$ is not empty.
The link between problems (\ref{eq:pdirect0}) and (\ref{eq:pdirect}) is far from trivial. \del{The decision variables
$(\mu,\mu_T)$  in (\ref{eq:pdirect}) can be viewed as a superposition of limiting objects arising from problem (\ref{eq:pdirect0})}. It is possible to construct problems for which measures considered in problem (\ref{eq:pdirect}) do not arise in this way which may introduce spurious minimizers which are far from classical trajectories of problem (\ref{eq:pdirect0}), \add{see for example \cite[Appendix C]{henrion2014convex}}. These problems are usually overly constrained and not physically relevant, and in most practical settings, we have
\[
	p^*(\delta_z) = v_0(z)\qquad \forall \,z \in \spt\:\mu_0,
\]
which we could see as an assumption on the inverse problem data. \add{In this constrained setting, sufficient conditions for this property to hold are those that ensure the applicability of the Filippov-Wa\.{z}ewski Theorem, see \cite{frankowska2000filippov} and the discussion around \cite[Assumption I]{gaitsgory2009linear}, \cite[Assumption 2]{henrion2014convex}, \cite[Assumption 1]{henrion2014linear}. Under such sufficient conditions, it can be shown using \cite[Theorem  2.3]{vinter1993convex} that the equality holds. However, as we argued in the introduction, the link between (\ref{eq:pdirect0}) and (\ref{eq:pdirect}) is much less problematic when considering inverse optimality. The main reason is that we consider that the input of the inverse problem is a measure, which is therefore given and fixed. It could arise as in (\ref{eq:supperpos2}) but not necessarily (see Figure \ref{fig:directIllustr}). We would like to emphasize the following:
\begin{itemize}
	\item if the input occupation measure does not satisfy (\ref{eq:supperpos2}), then it does not make sense to consider (\ref{eq:pdirect0}) as a basis for inverse optimality since the input of the problem itself is more general than the classical controls considered in (\ref{eq:pdirect0}). In this case, it is more relevant to focus on (\ref{eq:pdirect}) only.
	\item if the input occupation satisfy (\ref{eq:supperpos2}), then, the analysis is still valid. In this case, since the input of the inverse problem involves classical controls, {the question of the link between (\ref{eq:pdirect0}) and (\ref{eq:pdirect}) is a real issue for direct optimal control}. {But in the context} of inverse optimality, a partial answer is given \textit{a posteriori} by Corollary \ref{cor:noGap}. It is shown that, even in this case, considering (\ref{eq:pdirect}) as a basis for inverse optimality does not allow to identify Lagrangians for which there is a gap between (\ref{eq:pdirect0}) and (\ref{eq:pdirect}) for all considered initial conditions in $\spt \; \mu_0$, except for a $\mu_0$-negligible subset.
\end{itemize}

}

 \add{{For these reasons we adopt the following convention\\
 
\begin{center}
	{\it All our analysis refers to direct control problems of the form of (\ref{eq:pdirect}). }
\end{center}}}
\vspace{0.2cm}
{and the link with (\ref{eq:pdirect0}) (when it makes sense) will be \textit{a posteriori} justified by Corollary \ref{cor:noGap}: }
The corresponding conic dual can be written as
\begin{equation}\label{eq:pdirectDual}
	\tag{hjb}
\begin{array}{ll@{\;}l@{}}
	d^*(\mu_0) :=&  \displaystyle\sup_{v, w} & \left\langle \mu_0, v\right\rangle - w T_M\\
	& \mathrm{s.t.} & l + w + \mathrm{grad}\:v \cdot f \in \mathcal{C}_+(X \times U)\\
	&&-v \in \mathcal{C}_+(X_T),\\
          && w \geq 0,\\
	&&v \in \mathcal{C}^1(X), \:\: w \in {\mathbb R}.
\end{array}
\end{equation}
The first two constraints $l + w + \mathrm{grad}\:v \cdot f \in \mathcal{C}_+(X \times U)$ and $-v \in \mathcal{C}_+(X_T)$ of 
(\ref{eq:pdirectDual})
are relaxations of the well-known Hamilton-Jacobi-Bellman (HJB) sufficient condition of optimality \cite{athans1966optimal,bardi2008optimal}. 
Conic duality provides the following link between the problems (\ref{eq:pdirect}) and (\ref{eq:pdirectDual}).
\begin{lemma}
	\label{lem:duality}
	The infimum in (\ref{eq:pdirect}) is attained and there exists a maximizing sequence in (\ref{eq:pdirectDual}). In addition, for any feasible primal pair $(\mu, \mu_T)$ and any sequence of dual variables $v_k \in \mathcal{C}^1(X)$ and $w_k \in {\mathbb R}$,
$k \in \NN$, the following assertions are equivalent
	\begin{itemize}
		\item $(\mu, \mu_T)$ is optimal for (\ref{eq:pdirect}) and $(v_k, w_k)_{k\in\NN}$ is a maximizing sequence for (\ref{eq:pdirectDual});
		\item strong duality:
			\begin{equation}
				\label{eq:pdirectSlack1}
				\left\langle \mu_0, v_k\right\rangle - w_k T_M \limk \left\langle \mu, l\right\rangle;
			\end{equation}

		\item complementarity:
			\begin{equation}
				\label{eq:pdirectSlack2}
				\begin{array}{rcl}
					w_k(\left\langle\mu,1\right\rangle - T_M) &\limk& 0,\\
					\left\langle \mu, l + w_k + \mathrm{grad}\:v_k \cdot f \right\rangle &\limk& 0,\\
					\left\langle \mu_T, v_k \right\rangle &\limk& 0.
				\end{array}
			\end{equation}
	\end{itemize}
\end{lemma}

\begin{proof}
	We only sketch the proof here, for more details see \cite{lasserre2008nonlinear}. Observe that (\ref{eq:pdirect}) is feasible thanks to Assumption \ref{ass:feasibility} and (\ref{eq:pdirectDual}) is feasible with $w = \max(-\min_{X \times U} l, 0)$ and $v = 0$. Moreover, the cone $\{(\mathrm{div}f \mu + \mu_T, \left\langle \mu, 1\right\rangle, \left\langle \mu, l_0\right\rangle) : \mu \in \mathcal{M}_+(X \times U), \:\mu_T \in \mathcal{M}_+(X_T)\}$ is closed 
	for the weak topology $\sigma(\mathcal{M}_+(X)\times R^2,C_+(X)\times R^2)$ (by using Banach-Alaoglu's Theorem).
Therefore there is no duality gap between (\ref{eq:pdirect}) and (\ref{eq:pdirectDual}) and the optimum is attained in the primal, see e.g. \cite[Theorem IV.7.2]{alexander2002course}. Condition (\ref{eq:pdirectSlack1}) is just a reformulation of strong duality in this context. Equivalence with (\ref{eq:pdirectSlack2}) follows by noticing that for any primal feasible pair $(\mu, \mu_T)$ and dual feasible pair $(v, w)$, 
\[
\begin{array}{lcl}
		\left\langle \mu, l\right\rangle &
		= &\left\langle \mu, l\right\rangle - \left\langle \mathrm{div} f \mu + \mu_T - \mu_0, v\right\rangle\\
		&\geq &\left\langle \mu, l\right\rangle - \left\langle \mathrm{div} f \mu + \mu_T - \mu_0, v\right\rangle + w(
\left\langle \mu,1 \right\rangle - T_M) \\
		&=&\left\langle \mu, l + w +  \mathrm{grad}\:v \cdot f \right\rangle - \left\langle \mu_T, v\right\rangle + \left\langle \mu_0, v\right\rangle -  w T_M \\
		&\geq& \left\langle \mu_0, v\right\rangle - w T_M.
\end{array}
\]
\end{proof}
\begin{remark}
	If the Lagrangian $l$ is strictly positive on $X \times U$, then Lemma \ref{lem:duality}  holds without the constraint $\left\langle \mu,1 \right\rangle \leq T_M$ and without the dual variable $w$. 
\end{remark}

\section{Inverse optimal control}
\label{sec:iocp}
Given a ``set'' of trajectories and model constraints,  the inverse problem of optimal control consists of finding
a Lagrangian for which the trajectories are optimal. Thanks to the framework exposed in the previous section, it is now easy to define what is a solution to the inverse optimal control problem. 

$\bullet$ Firstly, the ``set'' of trajectories will be represented by measures satisfying Liouville equation (\ref{eq:liouville2}) which are part of the data of the inverse problem. 

$\bullet$ Secondly, a Lagrangian $l$ solution to the inverse problem is a continuous function such that $(\mu, \mu_T) \in \mathrm{OCP}(l, \mu_0, T)$ for some $T$ such that $\mathrm{OCP}(l, \mu_0, T)$ is feasible. 

In this section, we propose a rigorous definition of inverse optimality and prove an equivalence result between direct and inverse optimality. To do so, we use Lemma \ref{lem:duality} which ensures that $\mathrm{OCP}(l, \mu_0, T)$ is non empty as long as $T \leq T_M$. Furthermore, it provides a certificate of (sub)optimality. 

\subsection{What is a solution to the inverse optimal control problem?}
We can now formally define what is meant by a solution to the inverse optimal control problem:
\begin{definition}[IOCP and IOCP$_{\epsilon}$]
	\label{def:solution}
	For $\epsilon >0$, given measures $\mu \in \mathcal{M}_+(C\times U)$ and $\mu_T \in \mathcal{M}_+(X_T)$ such that
$\mathrm{div}f \mu + \mu_T \in \mathcal{M}_+(X_0)$, denote by  $\mathrm{IOCP}_\epsilon(\mu, \mu_T)$ the  set of $\epsilon$-optimal solutions to the inverse optimal control problem, namely the set of functions $l \in \mathcal{C}(X \times U)$ such that there exists a function $v \in \mathcal{C}^1(X)$ satisfying
\[
\begin{array}{rcl}
		\left\langle \mu, l +\mathrm{grad}\:v\cdot f\right\rangle &\leq& \epsilon,\\
		l + \mathrm{grad}\:v\cdot f + \epsilon & \in & \mathcal{C}_+(X \times U),\\
		\left\langle \mu_T, v\right\rangle &\geq & -\epsilon,\\
		-v & \in & \mathcal{C}_+(X_T).
\end{array}
\]
Then the set $\mathrm{IOCP}(\mu,\mu_T)$ of solutions to the inverse optimal control is defined by:
\[\mathrm{IOCP}(\mu,\mu_T)\,:=\,\{l\in \mathcal{C}(X \times U):\:  l \in \mathrm{IOCP}_{\epsilon}(\mu,\mu_T)\quad \forall\epsilon > 0\,\}.\]
\end{definition}

\add{Intuitively, Definition \ref{def:solution} states that we can find differentiable suboptimality certificate for any arbitrary precision (see in Remark \ref{rem:aproxSolution}). In addition, the positivity constraint on $l + \mathrm{grad}\:v\cdot f + \epsilon$ ensures that these certificates provide lower bounds on the value of the direct problem (\ref{eq:pdirect0}) for arbitrary initial conditions, even not in $\spt\; \mu_0$}. The main motivation behind this definition of inverse optimality is the following:

\begin{theorem}
	\label{th:solution}
Given $\mu \in \mathcal{M}_+(C\times U)$ and $\mu_T \in \mathcal{M}_+(X_T)$, the set
	$\mathrm{IOCP}(\mu, \mu_T)$ is a convex cone, closed for the supremum norm. Moreover, the following two assertions are equivalent:
	\begin{itemize}
		\item $l \in \mathrm{IOCP}(\mu, \mu_T)$, $\mathrm{div}f \mu + \mu_T = \mu_0 \in \mathcal{M}_+(X_0)$;
		\item $\exists\, T > \left\langle \mu,1\right\rangle$, $(\mu,\mu_T) \in \mathrm{OCP}(l, \mu_0, T)$.
	\end{itemize}
\end{theorem}

\begin{proof}
Convexity follows from convexity of the constraints of Definition \ref{def:solution}. There exists a constant $K$ such that for any pair $(l_0, v_0)$ that satisfies constraints of Definition \ref{def:solution} for a certain $\epsilon>0$, then it holds for any Lagrangian $l$ that
$\left\langle \mu, l + \mathrm{grad}\:v_0 \cdot f\right\rangle \leq \epsilon + K \|l - l_0\|_\infty$ and
$l + \mathrm{grad}\:v_0\cdot f \geq -\epsilon - K \|l - l_0\|_\infty$ on $X\times U$,
which is sufficient to prove closedness.
	
	For the first implication, suppose that $l \in \mathrm{IOCP}(\mu, \mu_T)$ and $\mathrm{div} f \mu + \mu_T = \mu_0$. Then for any $T > \left\langle \mu,1\right\rangle$, the pair $(\mu, \mu_T)$ is feasible for $\mathrm{OCP}(l, \mu_0, T)$. Lemma \ref{lem:duality} holds and the definition of $\mathrm{IOCP}(\mu, \mu_T)$ allows to construct a dual sequence that is feasible for $\mathrm{OCP}(l, \mu_0, T)$ and that satisfies the complementarity condition with the pair $(\mu, \mu_T)$.
		
	We now turn to the last implication. Suppose that $(\mu, \mu_T) \in \mathrm{OCP}(l, \mu_0, T)$ and $\left\langle \mu,1\right\rangle < T$. In particular, $(\mu, \mu_T)$ is feasible for $\mathrm{OCP}(l, \mu_0, T)$ and Lemma \ref{lem:duality} holds. Consider the dual maximizing sequence $(v_k,w_k)_{k\in\NN}$ given by Lemma \ref{lem:duality}. Complementarity ensures that $\lim_{k\to\infty}w_k=0$. Furthermore, it holds that 
$\lim_{k\to\infty} \left\langle \mu, l +\mathrm{grad}\:v_k\cdot f \right\rangle=0$,
$l + \mathrm{grad}\:v_k\cdot f {\geq} - w_k$ on $X\times U$,
$\lim_{k\to\infty}\left\langle \mu_T, v_k\right\rangle=0$ and
$v_k \leq 0$ on $X_T$ which shows that $l \in \mathrm{IOCP}(\mu, \mu_T)$.
\end{proof}
\begin{remark}
	\label{rem:aproxSolution}
	{\rm Another motivation behind Definition \ref{def:solution} of inverse optimality is the following. Suppose that $l \in \mathrm{IOCP}_\epsilon(\mu, \mu_T)$, then for any $T \geq \left\langle \mu,1\right\rangle$, $(\mu, \mu_T)$ is close to optimal for the problem $\mathrm{OCP}(l, \mu_0, T)$. Indeed, suppose that $(\tilde{\mu}, \tilde{\mu}_T) \in \mathrm{OCP}(l, \mu_0, T)$. 
Then there exists $v$ such that 
\begin{eqnarray*}
\left\langle \mu, l + \mathrm{grad}\:v\cdot f\right\rangle &=&\left\langle \mu, l\right\rangle + \left\langle \mu_T - \mu_0, v\right\rangle \leq \epsilon\\
\left\langle \tilde{\mu}, l +\mathrm{grad}\:v\cdot f\right\rangle &=& \left\langle \tilde{\mu}, l\right\rangle + \left\langle \tilde{\mu}_T - \mu_0, v\right\rangle\geq - T \epsilon,\end{eqnarray*}
and $-\left\langle \mu_T, v\right\rangle \leq \epsilon$ as well as 
$\left\langle\tilde{\mu}_T, v\right\rangle \leq \epsilon$.
In addition, $\left\langle \mu, l\right\rangle \leq \left\langle \mu_0, v\right\rangle + 2 \epsilon$ and $\left\langle \tilde{\mu}, l\right\rangle \geq \left\langle \mu_0, v\right\rangle - T \epsilon$. Therefore
$\left\langle \mu, l\right\rangle \geq \left\langle \tilde{\mu}, l\right\rangle \geq \left\langle \mu, l\right\rangle - (T + 2) \epsilon$.
}\end{remark}

\begin{remark}
{\rm At first sight the introduction of $T$ in Theorem \ref{th:solution} may look artificial whereas in fact it carries important information. The second part in the equivalence states that $(\mu, \mu_T)$ is a solution to some direct problem and does not saturate  one of the constraints. This allows to avoid direct problems for which, for any value of $T$, any solution would saturate the constraint on the mass of the occupation measure; for example this happens in direct problems with free terminal time tending to infinity. Such problems should be avoided since then an occupation measure with finite mass  cannot be optimal. Given a Lagrangian $l$, there is no guarantee that
there exists a triplet $(\mu, \mu_T, T)$ which satisfies the second point of Theorem \ref{th:solution}. However, checking that a Lagrangian $l$ meets our criterion for inverse optimality ensures that this is the case.}
\end{remark}

\add{An interesting corollary is that if the input of the optimal control is given by classical trajectories, then inverse optimality ensures that the value of (\ref{eq:pdirect0}) is attained by classical trajectories for almost all the initial values considered. This leaves {aside} many of the {technical issues when} working with classical trajectories for direct optimal control.} 

\add{\begin{corollary}\label{cor:noGap}
	{\rm If $l \in \mathrm{IOCP}(\mu, \mu_T)$ and $(\mu, \mu_0, \mu_T)$ is a superposition of classical trajectories as defined in equation \ref{eq:supperpos2} in Section \ref{sec:OM}, then $\mu_0$-a.a. (almost all) of these trajectories must be optimal for the corresponding direct problem. In particular, $v_0(z)$ given by (\ref{eq:pdirect0}) is attained and there is no relaxation gap between (\ref{eq:pdirect0}) with initial condition $z$ and (\ref{eq:pdirect}) with initial measure $\delta_z$ for $\mu_0$-a.a. initial conditions $z$ in $\spt\:\mu_0$. }
\end{corollary}}

\add{As a consequence, the focus on (\ref{eq:pdirect}) instead of (\ref{eq:pdirect0}) in Section \ref{sec:OCP} is \textit{a posteriori} justified by {Corollary \ref{cor:noGap}.} The question of absence of such a gap for initial conditions $z \not \in \spt\; \mu_0$ cannot be treated by this approach. This question is much less relevant for inverse optimality since it does not involve initial conditions that are related to input data of the inverse problem.}
\subsection{Applications to inverse optimality}
We claim that Definition \ref{def:solution} is a powerful tool to analyze inverse optimality in the context of optimal control. To go beyond Theorem \ref{th:solution}, we next describe results and comments that stem from Definition \ref{def:solution} of inverse optimality.

\subsubsection{How big is the space of solutions to the inverse problem?}
\label{sec:searchSpace}
Theorem \ref{th:solution} justifies the idea that if trajectories realize the minimum of some optimal control process then the corresponding Lagrangian meets our criterion. This requirement is necessary for any ``inverse problem" (and not only for
inverse optimal control). However in general there could be many candidate solutions as illustrated in this section. In what follows, we assume that the triplet $(\mu, \mu_0, \mu_T)$ satisfies Liouville's equation (\ref{eq:liouville2}).

\paragraph{Conserved values.} Suppose that there exists a function $g \in \mathcal{C}(X \times U)$ such that $g(x, u) = 0$ for all $(u, x) \in \spt\:\mu$. Then $g^2 \in \mathrm{IOCP}(\mu, \mu_T)$. In practical examples there might be many such conserved values. For instance this is the case when $x$ or $u$ or both remain on a manifold or when there exists a continuous mapping $x \to u(x)$. 

\paragraph{Total variations.} Consider any function $g \in \mathcal{C}^1(X)$. All Lagrangians of the form $l = \mathrm{grad}\:g\cdot f$ belong to $\mathrm{IOCP}(\mu, \mu_T)$, independently of $(\mu, \mu_0, \mu_T)$.

\paragraph{Convex conic combinations and uniform limits of solutions.} As stated in Theorem \ref{th:solution}, the set of solutions to the inverse problem is a convex cone, closed for the supremum norm. For example, let $g \in \mathcal{C}^1(X )$ and consider
a Lagrangian $l \in \mathcal{C}(X \times U)$. Then
$\mathrm{OCP}(l, \mu_0, T) = \mathrm{OCP}(l +\mathrm{grad}\:g\cdot f, \mu_0, T)$ for every $T>0$
and therefore both Lagrangians $l$ and $l + \mathrm{grad}\:g\cdot f$ are solutions to the inverse optimal control problem.

All the above examples illustrate that many solutions to the inverse problem may exist. Although these solutions are valid from a theoretical point of view, they do not correspond to what is commonly expected from a solution. Indeed, they do not arise from an optimal physical process that would have generated trajectories, but rather from mathematical artifacts. 

\subsubsection{How does the direct problem affect the space of solutions to the inverse problem?}
Intuitively, the more information is contained in $(\mu, \mu_T, \mu_0)$, the smaller is the space of solutions to the inverse problem. We next discuss two factors that impact the size of $\mathrm{IOCP}(\mu, \mu_T)$.

\paragraph{Direct problem constraints.} Denote by $\mathcal{K}^1$ (resp. $\mathcal{K}^2$) the feasible set of problem (\ref{eq:pdirect}) and by $\mathrm{IOCP}^1(\mu, \mu_T)$ (resp. $\mathrm{IOCP}^2(\mu, \mu_T)$) the set of solutions to the inverse problem 
(as described in Definition \ref{def:solution})  when the state, control and dynamical constraints are given by $(X^1, U^1, f^1)$ (resp. $(X^2, U^2, f^2)$). 

If $\mathcal{K}^1 \subset \mathcal{K}^2$ then $\mathrm{IOCP}^2(\mu, \mu_T) \subset \mathrm{IOCP}^1(\mu, \mu_T)$. In other words, there is  a kind of duality between the space of feasible solutions for the direct problem and the space of solutions to the inverse problem. An extreme instance is when the feasible space of the direct problem is a singleton ($f$ does not depend on the control $u$), in which case any Lagrangian is a solution to the inverse problem.

\paragraph{Range of the occupation measure.} Suppose that $(\mu, \mu_0, \mu_T) = (\mu^1, \mu^1_0, \mu^1_T) + (\mu^2, \mu^2_0, \mu^2_T)$ where $\mathrm{div}f \mu^1 + \mu^1_T = \mu^1_0$ and $\mathrm{div}f \mu^2 + \mu^2_T = \mu^2_0$. 
Then $\mathrm{IOCP}(\mu, \mu_{T})\subset \mathrm{IOCP}(\mu^1, \mu^1_T)$. As a consequence, maximizing the support of the initial measure $\mu_0$ reduces the space of solutions to the inverse problem. When the occupation measure $\mu$ is a superposition of trajectories as detailed in Section \ref{sec:OM}, the larger is the ``space" occupied by trajectories, the smaller is the space of potential solutions to the inverse problem. 

\subsubsection{A toy example of quantitative well-definedness analysis}
\label{sec:toyExample}
To illustrate the proposed framework we consider a simple uni-dimensional example. We emphasize that his example is very simple in the sense that the direct problem is easy. However, inspecting the solution of the inverse problem leads to non trivial behaviors.
Let $X  = [-1, 1]$  with $X_T = \{0\}$ and let $f(x,u) = u$ with $U = [-1, 1]$. Consider the family of Lagrangians $\mathcal{F} = \{l_\alpha \colon u\mapsto 1 + \frac{\alpha}{2}u^2,\, \alpha \geq 0\}$. Suppose that we are given a triplet $(\mu, \mu_0, \mu_T)$ which consists of a superposition of trajectories as described in Section \ref{sec:OM}. We wish to find a candidate Lagrangian in the family $\mathcal{F}$. Then we have the following alternatives.
\begin{itemize}
	\item[1] $\mathcal{F} \cap \mathrm{IOCP}(\mu,\mu_T) = \emptyset$.
	\item[2] $\mathcal{F} \cap \mathrm{IOCP}(\mu,\mu_T) = \mathcal{F}$.
	\item[3] $\mathcal{F} \cap \mathrm{IOCP}(\mu,\mu_T)$ is a singleton, $\{l_\alpha\}$, $\alpha  > 2$.
	\item[4] $\mathcal{F} \cap \mathrm{IOCP}(\mu,\mu_T) = \{l_\alpha,\, 0 \leq \alpha \leq 2\}$ .
\end{itemize}
We should comment on case 1 latter. If the support of $\mu$ is empty, which means that $\mu_0 = \mu_T = \delta_x$, then we are in case 2. Assume now that the support of $\mu$ is non-empty and we are not in case 1. Then there exists $\alpha \geq 0$ such that $l_\alpha \in \mathrm{IOCP}(\mu,\mu_T)$. Consider a sequence of decreasing positive numbers $\epsilon_k \to 0$ and the corresponding certificates functions $v_k$ that allow to verify that $l_\alpha \in \mathrm{IOCP}_{\epsilon_k}(\mu, \mu_T)$. Since $\mu_T = \delta_0$, we may assume (up to an addition) that $v_k(0) = 0$ which simplifies the problem. In addition, one must have $1 + \frac{\alpha}{2} u^2 + v_k'(x) u \to 0$, $\mu$ almost every where (recall that the support of $\mu$ is non empty and this concerns a non empty subset of $X$ and $U$). Furthermore, for any $x$, one must have $1 + \frac{\alpha}{2} u^2 + v_k'(x) u + \epsilon_k = \left( \sqrt{\frac{\alpha}{2}} u + \frac{v'(x)}{\sqrt{2\alpha}} \right)^2 + 1 + \epsilon_k - \frac{v'(x)^2}{2 \alpha}\geq 0$, for $u \in [-1, 1]$. 
\begin{itemize}
	\item Suppose that $\alpha > 2$. Then for any $k$, $|v'_k(x)| \leq 1 + \frac{\alpha}{2} + \epsilon_k$. Since $\epsilon_k$ goes to $0$ and $\alpha > 2$, for $k$ sufficiently large, has $\frac{|v'_k(x)|}{\alpha} \leq 1$. Taking $u = -\frac{v'_k(x)}{\alpha}$ gives $1 + \frac{\alpha}{2} u^2 + v_k'(x) u + \epsilon_k \geq 1 + \epsilon_k - \frac{v'(x)^2}{2 \alpha}\geq 0$. It must hold $\mu$ almost everywhere that $ 1 + \epsilon_k - \frac{v'(x)^2}{2 \alpha} \to 0$ and $|u| = \sqrt{\frac{2}{\alpha}}$.
	\item Suppose that $0 \leq\alpha\leq 2$. It must hold $\mu$ almost every where that $v'_k(x)u \to - 1 - \frac{\alpha}{2}u^2$. This implies that $u \neq 0$ and $|v'(x)| \to \frac{1+\frac{\alpha}{2}u^2}{|u|}$, $\mu$ almost every where. It can be verified that for $\alpha \leq 2$, this is a strictly decreasing function of $|u|$ for $|u| \leq 1$. Therefore, it holds that $\lim\inf|v'_k(x)| \geq 1 + \frac{\alpha}{2}$, $\mu$ almost every where. Since we have $|v'_k(x)| \leq 1 + \frac{\alpha}{2} + \epsilon_k$, it holds that $|v'_k(x)| \to 1 + \frac{\alpha}{2}$ and therefore $|u| = 1$, $\mu$ almost everywhere. In this case, it is easy to construct alternative sequences $\tilde{v}_k = \frac{2 + \tilde{\alpha}}{2 + \tilde{\alpha}} v_k$ for $0 \leq \tilde{\alpha} \leq 2$ to show that $l_{\tilde{\alpha}}$ is also a member of $\mathrm{IOCP}(\mu, \mu_T)$. 
\end{itemize}
To conclude, we are in case $1$ when the trajectories that generate $\mu$ are not optimal with respect to any Lagrangian in $\mathcal{F}$, in particular when $|u|$ is not $\mu$ almost everywhere constant. If this is not the case and $\mu$ is not degenerate, we have a unique solution or a set of solutions depending on $\mu$ and its relation with the constraint on $u$.

\section{Practical inverse control}
\label{sec:practical}
As discussed in Section \ref{sec:searchSpace}, the space of solutions to the inverse problem can be very large. Many of these solutions are of little interest for practitioners because they lack some physical meaning. However, from a formal point of view 
``valid" solutions exist and ideally they should be the only solutions of
a {\it practical inverse optimal control problem} to be defined. 

One may invoke some heuristics to reduce the space of solutions and to enforce prior knowledge in the treatment of the inverse problem. This is commonly achieved by imposing constraints on the candidate Lagrangian solution. Such heuristics include :
\begin{itemize}
	\item restricting the dependence  on certain variables;
	\item shape conditions (e.g., convexity);
	\item conic constraints such as positivity;
	\item parametric constraints (e.g., considering a finite dimensional family of candidate Lagrangians);
	\item constraints relating the dependence between the candidate Lagrangian and the corresponding value function.
\end{itemize}

Notice that Definition \ref{def:solution} refers to the large class  of continuous Lagrangians with conic constraints. From a theoretical perspective, this allows to characterize inverse optimality in full generality. However this is not amenable to numerical computation yet and so we also describe tractable numerical approximations in the context of inverse optimality.

Finally, according to Definition \ref{def:solution}, the input of the inverse problem is an occupation measure. Again, this is a convenient tool for theoretical purposes but in most practical cases such an occupation measure is not available.
In fact, roughly speaking, only some realizations of an experiment are available and
these realizations form a data set which is an approximation of an hypothetical occupation measure. Therefore in practice
the input of the inverse problem is only an approximation of an ideal input,
and correctness of this approximation is justified under certain experimental assumptions at the end of this section.

\subsection{Normalization}
The trivial Lagrangian is solution to the inverse problem independently of the input occupation measures. As we have seen in Section \ref{sec:searchSpace}, total variations share the same property. Even though these are solutions to the inverse problem, it is important to avoid them in practice because they do not depend on the input occupation measure and therefore carry no information about it. As illustrated in Example \ref{sec:toyExample}, one way to avoid these spurious solutions is to consider only very restricted families of Lagrangians that cannot contain such solutions. This might be quite restrictive in practice and therefore we provide an alternative.
We need the following assumption:

\begin{assumption}[Finite time controllability]
	\label{as:controlability}
	There exists $T > 0$ and a compact set $\tilde{X} \subset X$ with smooth boundary and $\mu_0(\tilde{X}) > 0$, such that for any $x_1, x_2 \in \tilde{X}$, there exists $s \in [0, T]$, a bounded function $u : [0,s] \to U$ and an absolutely continuous trajectory $x : [0,s] \to X$ such that $x(0) = x_1$, $x(s) =x_2$ and $\dot{x}(t) = f(x(t), u(t))$ for all $t \in [0, s]$.
\end{assumption}

 Under assumption \ref{as:controlability} we have the following result.
\begin{proposition}
	\label{prop:normalization}
	If  in Definition \ref{def:solution} one includes the normalization
	\[
		\left|1 - \int_{\tilde{X} \times U}l + \mathrm{grad}\:v\cdot f\right| = \epsilon
	\]
	then $0 \not\in \mathrm{IOCP}(\mu, \mu_T)$.
\end{proposition}
\begin{proof}
	This is due to the following contradiction. Suppose that $0 \in \mathrm{IOCP}(\mu, \mu_T)$.
	Choose a decreasing sequence $\epsilon_k\to 0$ as $k\to\infty$, and construct a sequence $v_k$ of differentiable functions that satisfy conditions of Definition \ref{def:solution} for the chosen $\epsilon_k$.  Then
$\left\langle\mu_0,v_k \right\rangle = \left\langle\mu_T,v_k \right\rangle - \left\langle \mu, \mathrm{grad}\:v_k\cdot f\right\rangle = O(\epsilon_k) $.
Because of Assumption \ref{ass:feasibility}, $v_k(x_0) \leq T_M\epsilon_k$ for every $x_0 \in X_0$.
Therefore, by integration, $\max(\int_{\tilde{X}} v_k \mu_0, \int_{X \setminus \tilde{X}} v_k \mu_0) \leq O(\epsilon_k)$. Furthermore, $\int_{\tilde{X}} v_k \mu_0+ \int_{X \setminus \tilde{X}} v_k \mu_0  = \left\langle\mu_0, v_k \right\rangle = O(\epsilon_k)$. Finally, $$\max\left(\int_{\tilde{X}} v_k \mu_0, -\int_{\tilde{X}} v_k \mu_0\right) \leq O(\epsilon_k)$$ and $\lim_{k\to\infty}\int_{\tilde{X}} v_k=0$.
	
	In addition, by Assumption \ref{as:controlability}, we also have 
	$|v_k(x_1) - v_k(x_2)| \leq T \epsilon_k$ for every $x_1, x_2 \in \tilde{X}$.
	Since $\mu_0(\tilde{X}) > 0$, this implies that $v_k\to 0$ uniformly on $\tilde{X}$, as $k\to\infty$. In addition,
	$\int_{\tilde{X} \times U}\mathrm{grad}\:v_k\cdot f \to 1$ as $k\to\infty$.
	Next, with the polynomial vector field 
	\[x\mapsto g(x) = \int_U f(x,u) du,\]
	Stokes' Theorem yields
	\begin{align*}
		\int_{\tilde{X}} \text{div}(v_k(x)g(x)) &= \int_{\tilde{X}} \text{div}(g(x)) v_k(x) + \mathrm{grad}\:v_k(x)\cdot g(x)\\
		&=  \int_{\tilde{X}} \text{div}(g(x)) v_k(x) + \int_{\tilde{X} \times U} \mathrm{grad}\:v_k(x)\cdot f(x,u)\\
		&= \int_{\partial \tilde{X}} v_k(x) g(x) \cdot \vec{n}(x)
	\end{align*}
          where $\vec{n}(x)$ is the outward pointing normal to the boundary at $x$.
	Because of the uniform convergence of $v_k\to 0$ on $\tilde{X}$ as $k\to\infty$, and boundedness of $g$ and $\mathrm{div}\:g$ on $X$, the left-hand side converges to $1$ while the right-hand side converges to $0$, which is a contradiction.
\end{proof}


\begin{remark}
{\rm 	The conditions on $\tilde{X}$ may be relaxed. Indeed, the only important point is to be able to apply Stokes's Theorem.
	In particular, the set $\tilde{X}$ could  be a box or an open set whose boundary does not have too many non-smooth points, see e.g. \cite[Theorem III.14A]{whitney1957geometric}.
}\end{remark}
\begin{remark}
{\rm 	The normalization given in Proposition \ref{prop:normalization} obviously ensures that 
	 Lagrangians in the form of a total variation are excluded.
}\end{remark}
Assumption \ref{as:controlability} may look very strong regarding the result of Proposition \ref{prop:normalization}. However
the next example shows that it cannot be excluded.

\begin{example}
	Consider the direct control problem with $X = U = [0, 1]$, $X_T = \{1\}$, $T_M = 1$ and $f(x, u) = u$. These data are obviously not compatible with Assumption \ref{as:controlability}. Choose $\mu_0(dx) = dx$ and $l_0(x,u) = 1$  so that the couple $\mu(dx,du) := x dx \delta_1(du)$ and $\mu_T(dx) := \delta_1(dx)$, solves the problem
	\begin{equation}\label{eq:exDirect}
	\begin{array}{ll@{\;}l@{}}
		p^*(\mu_0) :=& \displaystyle\inf_{\mu, \mu_T} & \left\langle \mu, 1\right\rangle \\
		& \mathrm{s.t.} & \mathrm{div}f \mu + \mu_T = \mu_0,\\
	           &&\left\langle \mu,1 \right\rangle \leq  T_M,\\
						 &&\mu \in \mathcal{M}_+(X \times U),\\
						 &&\mu_T \in \mathcal{M}_+(X_T).
	\end{array}
	\end{equation}
	Indeed, for any differentiable $v$, 
	$\left\langle \mathrm{div} f \mu, v \right\rangle = \int_T^1 - v' f \mu(dx) = \int_T^1 v dx - v(1) = \left\langle \mu_0, v \right\rangle - \left\langle \mu_T, v \right\rangle.$
	Furthermore, the function $x\mapsto v^*(x): = 1 - x$ ensures that $(\mu,\mu_T)$ is an optimal solution. Indeed 
	$\left\langle \mu, 1 \right\rangle =  \left\langle \mu_0, v^* \right\rangle.$
	Consider a sequence of differentiable functions $(v_k)$, $k\in\mathbb{N}$, such that $v_k(x) = 0$ for $x \geq \frac{1}{k}$ and $v_k(x) = -(kx -1)^2$ otherwise. For $\tilde{X} = [0, 1]$, we have
$0 \leq \left\langle \mu, v_k' f \right\rangle = \int_T^\frac{1}{k} kx - k^2x^2dx = \frac{1}{6k} \to 0$,
$v_k'f \geq 0$, $\left\langle \mu_T, v_k \right\rangle = 0$,
$v_k(1) = 0$, $\int_{\tilde{X}\times U} v'_k f = v_k(1) - v_k(0) = 1$.
Therefore even if we enforce the normalization $\int_{\tilde{X}\times U} l + v'f = 1$, we cannot prevent the trivial Lagrangian $l = 0$ from
being an optimal solution to the inverse problem.
\end{example}

\begin{remark}
{\rm 	Regarding Proposition \ref{prop:normalization}, one could argue that simple linear constraints such as $l(0) =1$ or conic constraints such as $l > 0$ would be sufficient to avoid the trivial Lagrangian. However, this does not allow to avoid total variations which are equivalent to the trivial Lagrangian in terms of solutions to the direct problem.
}\end{remark}

Denote by $\mathrm{NIOCP}$ the subset of $\mathrm{IOCP}$ with the normalization constraint of Proposition \ref{prop:normalization} added to the constraints of Definition \ref{def:solution}. One important feature of this normalization is that it can be thought of as a way to intersect the cone of solutions to the inverse problem with an affine subspace. Therefore $\mathrm{NIOCP}$ is still closed and convex. Furthermore, if we restrict the set of candidate Lagrangians to be finite-dimensional, one may look for minimum norm-like solutions which will prove to be useful in numerical experiments. Indeed, $\mathrm{NIOCP}$ being closed and convex and all norms being equivalent in finite dimensions, optimization problems over $\mathrm{NIOCP}$, if bounded, have an optimal solution. Finally,
as $0 \not\in \mathrm{NIOCP}$, one may minimize any norm-like function to enforce specific prior structure and avoid the trivial Lagrangian.

\subsection{Polynomial approximation}
Until now, all the results that we have presented involve continuous and differentiable functions, in full generality. 
However, for practical computation one has to approximate such functions
and of course, {\it polynomials} are obvious natural candidates. But in our context they are also of particular interest for mainly three reasons:
\begin{itemize}
	\item for fixed degree, polynomials belong to finite-dimensional spaces and are therefore amenable to computation;
	\item when varying the degree, the class of polynomials is rich enough to approximate a wide class of functions;
	\item {\it Positivity Certificates} from real algebraic geometry allow to express positivity constraints in a computationally tractable way.
\end{itemize}
From now on, we make the following assumption:
\begin{assumption}
	\label{ass:poly}
	$f$ is a polynomial and $X$, $X_T$ and $U$ are basic semi-algebraic sets.
\end{assumption}
As proposed in Definition \ref{def:solution}, checking that a polynomial is a solution of the inverse problem involves the construction of a sequence of continuously differentiable functions. These functions can also be approximated by polynomials. In this section
we describe some tools required for such an approximation and we also prove the correctness of the approximations.

Let $g_1, \ldots , g_m \in \RR[z]$ be polynomials in the variable $z\in\RR^{n}$ and consider the basic semi-algebraic set
\[
G = \{z \in \RR^{n} : g_i(z) \geq 0,\, i=1,\ldots,m\}.
\]
Let $g_0=1$ and let $\Sigma^2\subset\mathbb{R}[z]$ denotes the set of sums-of-squares (SOS) polynomials,
i.e., $p\in\Sigma^2$ if it can be written as a sum of squares of other polynomials.
\begin{definition}
Let $Q_k(G)$ denote the convex cone of polynomials that can be written as
\[
		p = \sum_{i=0}^m s_i g_i, \quad s_i \in \Sigma^2,\, i = 0,1,\ldots,m,
\]
where the degree of $s_i g_i$, $i = 0,1, \ldots m$, is at most $2k$.
If $p \in Q_k(G)$ we say that $p$ has a Putinar positivity certificate.
\end{definition}
It is immediate to check that any element of $Q_k(G)$ is non-negative on $G$. A remarkable property of such certificates is that 
a partial converse is true. 

\begin{proposition}[\cite{putinar1993positive}]
	\label{prop:putinar}
	Suppose that the polynomial super-level set $\{x : g_i(x) \geq 0\}$ is compact for some $i=1,\ldots,m$.
 If $p>0$ on $G$ then  there exists $k \geq 0$ such that $p \in Q_k(G)$. 
\end{proposition}

Furthermore and importantly from a computational viewpoint, checking 
whether $p\in\Sigma^2$ reduces to checking whether a set of LMIs \cite{lasserre2010} involving the coefficients of $p$ has a solution. Therefore a more precise definition of inverse optimality in the context of polynomial Lagrangians is as follows (compare with Definition \ref{def:solution}).

\begin{definition}[polyIOCP and polyIOCP$_{\epsilon,k}$]
	\label{def:polySolution}
For $\epsilon >0$, given measures $\mu \in \mathcal{M}_+(C\times U)$ and $\mu_T \in \mathcal{M}_+(X_T)$ such that
$\mathrm{div}\:f\mu+\mu_T=\mu_0 \in \mathcal{M}_+(X_0)$, denote by  $\mathrm{polyIOCP}_{\epsilon,k}(\mu,\mu_T)$ 
the set of polynomials $l \in \RR[x,u]_{2k}$ (i.e., of degree at most $2k$) such that:
\[
\begin{array}{rcl}
		\left\langle \mu, l + \mathrm{grad}\:v\cdot f\right\rangle &\leq  & \epsilon,\\
		l + \mathrm{grad}\:v\cdot f + \epsilon & \in & Q_k(X \times U), \\
		\left\langle \mu_T, v\right\rangle &\geq & - \epsilon,\\
		- v &\in & Q_k(X_T),
\end{array}
\]
for some polynomial $v \in \RR[x]_{2k}$.

Denote also by $\mathrm{polyIOCP}(\mu, \mu_T)$ the set of polynomial solutions to the inverse optimal control problem: That is,
$l \in \mathrm{polyIOCP}(\mu, \mu_T)$ if for any $\epsilon >0$ there exists $k(\epsilon)$ such that
$l \in \mathrm{polyIOCP}_{\epsilon,k(\epsilon)}$.
\end{definition}

In other words, $\mathrm{polyIOCP}_{\epsilon,k}(\mu,\mu_T)$ is the set of polynomial $\epsilon$-solutions with degree bound $2k$, to the inverse optimal control problem.
The advantage of the previous definition, is that, provided that one has the possibility to compute the linear functionals 
$\mu\in\RR[x,u]^*$ and $\mu_T\in\RR[x]^*$, checking whether $l \in \mathrm{polyIOCP}_{\epsilon,k}(\mu, \mu_T)$  for $k$ and $\epsilon$ given reduces to solving a convex LMI problem \cite{lasserre2010}. Furthermore, under a compactness assumption, in the asymptotic regime this definition is equivalent to Definition \ref{def:solution}.

\begin{proposition}[Correctness of polynomial approximation] Suppose that one of the polynomials defining the basic semi-algebraic set $X$ (resp. $U$) has a compact super-level set. Then
	$$\mathrm{polyIOCP}(\mu, \mu_T) = \mathrm{IOCP}(\mu, \mu_T) \cap \RR[x,u].$$ 	
\end{proposition}
\begin{proof}
	The direct inclusion is trivial. For the reverse inclusion, suppose that $l \in \mathrm{IOCP}(\mu, \mu_T) \cap \RR[x,u]$. Fix $\epsilon >0$, and take for $v$ a certificate that	 $l \in  \mathrm{IOCP}_\frac{\epsilon}{3}(\mu, \mu_T)$ as given by Definition \ref{def:solution}. Since we consider compact sets in finite dimensional spaces, both $v$ and its gradient $\mathrm{grad}\:v$ can be simultaneously approximated uniformly by a polynomial up to an arbitrary precision. Therefore as
$f$ is bounded on $X \times U$,  there exists a polynomial $v_k$ of degree $k$ such that $\sup_{X} |v -v_k| \leq \frac{\epsilon}{3}$ and $\sup_{X\times U} |\mathrm{grad}\:v\cdot f - \mathrm{grad}\:v_k\cdot f| \leq \frac{\epsilon}{3}$. Hence $l + \mathrm{grad}\:v_k\cdot f + \epsilon \geq \frac{\epsilon}{3} > 0$ on $X \times U$ and $\frac{\epsilon}{3} - v_k \geq \frac{\epsilon}{3} > 0$ on $X_T$. Using Proposition \ref{prop:putinar}, there exists $k_1$ and $k_2$ such that $l + \mathrm{grad}\:v_k\cdot f + \epsilon \in Q_{k_1}(X \times U)$ and $\frac{\epsilon}{3} - v_k \in Q_{k_2}(X_T)$. Then
$l \in \mathrm{polyIOCP}_{\epsilon,k}(\mu, \mu_T)$ whenever $k\geq\max(k, k_1, k_2)$ and
finally, $l \in \mathrm{polyIOCP}(\mu, \mu_T)$ because $\epsilon$ was arbitrary (fixed).
\end{proof}
	
All properties of Lagrangians in Definition \ref{def:solution} hold for the Lagrangians in Definition \ref{def:polySolution}. For example, as stated in Remark \ref{rem:aproxSolution}, if $l \in \mathrm{polyIOCP}_{\epsilon,k}(\mu, \mu_T)$, then $(\mu, \mu_T)$ is close to optimal for $\mathrm{OCP}(l, \mu_0, T)$. In particular, if we the moments of $\mu$ and $\mu_T$ are available then the latter property can be checked numerically by solving a semi-definite program.

\subsection{Integral discretization}
\label{sec:discrete}
For practical numerical computation in the context of Definition \ref{def:polySolution}, we still must be able to integrate 
polynomials with respect to $\mu$ and $\mu_T$. This is easy provided that we know the moments of $\mu$ and $\mu_T$. However, exact computation of such moments cane be complicated in practice, especially when $\mu$ is a superposition of trajectories. Usually, data sets from experiments consist of samples of trajectories which can be seen as realizations of a random sampling process. 

In this section we first describe how the framework of occupation measures can formally describe the process of sampling trajectories and we justify the replacement of measures $(\mu, \mu_T)$ by their empirical counterparts when considering empirical samples as input data for inverse control problems. In the context of polynomial certificates in Definition \ref{def:polySolution}, this amounts to replacing the moments of the measures $(\mu, \mu_T)$ by their empirical counterparts. Consider the probability measure $\mu_0$ on $X_0$ and the measures $(\mu, \mu_T)$ in equations (\ref{eq:supperpos1}) and (\ref{eq:supperpos2}). One way to interpret these measures is to consider the following random process:
\begin{itemize}
	\item choose $z \in \spt\:\mu_0$ randomly,
	\item choose $t$ randomly uniformly on $[0, T_z]$,
	\item output $\xi = (x_z(t), u_z(t))$.
\end{itemize}
This defines a generative process for the random variable $\xi$. If the probability for an initial condition $z$ to belong
to a Borel set $A \subset X_0$ is given by
\begin{align*}
	&\PP[z \in A]\\
	=&p_0(A) := \frac{\int_A T_z \mu_0(dz)}{\int_{X_0} T_z \mu_0(dz)},
\end{align*}
then the probability for a trajectory $\xi$ to belong to a Borel hyperrectangle $(A,B) \subset X\times U$ is given by
\begin{align*}	
 &\PP[\xi \in A\times B] \\
 =& p_\mu(A \times B) := \int_{X_0} \left( \int_0^{T_z} I(x_z(t) \in A)I(u_z(t) \in B) dt\right) \mu_0(dz) =
\frac{\mu(A\times B)}{\mu(X\times U)}.
\end{align*}
A statistical model for points $(x_i, u_i),\, i=1, \ldots, n$, that are samples of trajectories, is to assume that we repeat the previous process $n$ times, independently. In this case, we say that the database $\mathcal{D} = \{(x_i, u_i)\}_{i=1}^n$ is made of independent realizations of a random variable with underlying distribution $p_\mu$. The process which generates the database being random, we write 
\[
	\{\xi_i\}_{i=1}^n \overset{i.i.d.}{\sim} p_\mu.
\]
to stress that all $\xi_i$ are independent and identically distributed (\textit{i.i.d.}) according to $p_\mu$ (they are independent copies of the same random variable). Similarly, $\mu_T$ can be seen as the probability distribution describing the following random process:
\begin{itemize}
	\item choose $z \in \spt\:\mu_0$ randomly according to $p_0$,
	\item output $x_z(T_z)$.
\end{itemize}
We now define what is an approximate solution to the inverse problem when the only information available about $\mu$ is a realization, $\mathcal{D} = \{(x_i, u_i)\}_{i=1}^n$, of a random process, $\{\xi_i\}_{i=1}^n \overset{i.i.d.}{\sim} p_\mu$.

\begin{definition}[Sampled-IOCP$_{\epsilon,k}$]
	\label{def:approxSolution}
For $\epsilon > 0$ and $\mathcal{D} = \{(x_i, u_i)\}_{i=1}^n$, let $\mathrm{Sampled-IOCP}_{\epsilon,k}$ 
be the set of  polynomials $l \in \RR[x,u]_{2k}$ such that :
	\begin{align*}
		&\frac{1}{n} \sum_{i=1}^n l(x_i, u_i) + \mathrm{grad}\:v(x_i)\cdot f(x_i, u_i) \leq \epsilon,\\
		&l + \mathrm{grad}\:v\cdot f  \in Q_k(X \times U), \\
		&- v \in Q_k(X_T),\\
		&v + \epsilon \in Q_k(X_T),\\
		&\int_{\tilde{X} \times U} l + \mathrm{grad}\:v\cdot f = 1.
	\end{align*}
for some polynomial $v \in \RR[x]_{2k}$. 
In other words, $\mathrm{Sampled-IOCP}_{\epsilon,k}$ 
is the set of polynomial $\epsilon$-optimal solutions (with degree bound $k$) of the sampled inverse optimal control problem., 
\end{definition}

One has replaced $\mu$ by its empirical counterpart, added the normalization of Proposition \ref{prop:normalization} and simplified other conditions; see also Remark \ref{rem:finalMeasure}. 
\begin{center}
{\it Importantly, membership in $\mathrm{Sampled-IOCP}_{\epsilon,k}(\mathcal{D})$ can be tested by
semi-definite programming.}
\end{center}

Using arguments from empirical processes and learning theory, one can quantify the price to pay for this discretization. Of course since we assume that the process that generates the data include some randomness, such a quantification holds probabilitically.
\begin{proposition}
	\label{prop:unifApprox}
	Suppose that $\mathcal{D} = \{(x_i, u_i)\}_{i=1}^n$ is a realization of the random process $\{\xi_i\}_{i=1}^n \overset{iid}{\sim} p_\mu$. Then there exist constants $K_1(X, U, k)$, $K_2(X, U, k)$ that only depend on ($X,U)$, and $k\in\mathbb{N}$, such that for any $l \in \mathrm{Sampled-IOCP}_{\epsilon,k}(\mathcal{D})$ and any $0 < \delta < 1$:
	\[
		l \in \mathrm{polyIOCP}_{\epsilon'(n),k}(\mu, \mu_T)\mbox{ with probability }1-\delta,\]
		where 
		\[\epsilon'(n) = \epsilon + \frac{1}{\sqrt{n}}\left(K_1(X, U, k) + K_2(X, U, k) \sqrt{\ln \frac{2}{\delta}}\right),
	\]
	and where the randomness comes from the realization of $\mathcal{D}$.
\end{proposition}
\begin{proof}
	Apply Lemma \ref{eq:finiteSampleBound}  of Appendix \ref{sec:statLearn} to the polynomial $l +\mathrm{grad}\:v\cdot f$
 to get a bound on $\left\langle p_\mu, l + \mathrm{grad}\:v\cdot f\right\rangle$.
\end{proof}
\begin{remark}\label{rem:finalMeasure}
{\rm 	The conditions detailed in Definition \ref{def:approxSolution} ensure that $v \leq 0$ on $X_T$ and $\left\langle \mu_T, v\right\rangle \geq - \epsilon$, for any terminal measure. This is done in order to avoid to deal with the terminal measure $\mu_T$ but other alternatives are possible. For example, when $X_T$ is a single point or a simple algebraic set one may enforce 
(as we do in Section \ref{sec:numerics}) $v =0$ on $X_T$ instead. Another possibility is to replace $\mu_T$ by its empirical counterpart. However in this case we need to provide a lower bound or add constraints on $v$ to obtain finite sample bounds as described in Proposition \ref{prop:unifApprox}.
}\end{remark}
Proposition \ref{prop:unifApprox} mixes arguments from measure theory and conic optimization with arguments from empirical process and statistical learning theory. The implication of this result is that for a fixed degree, provided that the sample size is big enough, with high probability we do not loose much by approximating $\mu$ by an empirical sample. 

$\bullet$ We would like to emphasize here that it is necessary to restrict the complexity of the class of functions in which the candidate Lagrangian is searched. Indeed otherwise for instance, for any fixed sample, the polynomial 
$(x,u)\mapsto \prod_{i=1}^n ||x - x_i||^2||u - u_i||^2$ belongs to $\mathrm{Sampled-IOCP}_{0, 4n}(\mathcal{D})$, but this clearly does not give much insight on the original control problem! 
The degree of the polynomial candidates is one among many possible measures of complexity. 
Furthermore, although the constants are likely to be sub-optimal, they give a sense of how fast the degree of the polynomial approximation may grow with respect to the sample size in order to maintain accurate approximations of $\mu$. 

\section{Numerical illustrations}
\label{sec:numerics}
Building on results of Section \ref{sec:practical}, we next provide illustrative numerical simulations. In order to fit in the framework of the previous section, we consider examples where $f$ is a polynomial and $X$, $U$ and $X_T$ are basic semi-algebraic sets. In addition, the input data of the inverse problem is given by a finite database: $\mathcal{D} = \{(x_i, u_i)\}_{i=1,\ldots,n}$. 
In the sequel, the candidate Lagrangians satisfy Definition \ref{def:approxSolution}. 

To compute such Lagrangians, the main idea is to solve 	an optimization problem with fixed $k\in\mathbb{N}$, where:
\begin{itemize}
\item $l\in\RR[x,u]_{2k},v\in\RR[x]_{2k}$ and $\epsilon>0$ are the decision variables,
\item $\epsilon$ is the criterion  to minimize,
\item $\mathrm{Sampled-IOCP}_{\epsilon, k}(\mathcal{D})$ is the projection on $(l,\epsilon)$ of the set of feasible solutions $(l,\epsilon,v)$.
\end{itemize}
In addition, we also include a sparsity inducing term in the criterion that will prove to be useful in numerical experiments.

\subsection{Numerical experiments}
\subsubsection{Problem formulation}
We consider the following optimization problem:
\begin{equation}\tag{iocp}\label{eq:primal1}
\begin{array}{l@{\;}l}
\displaystyle \inf_{l, v, \epsilon} &  \epsilon + \lambda ||l||_1\\
\mathrm{s.t.}	&\frac{1}{n} \sum_{i=1}^n l(x_i, u_i) + \mathrm{grad}\:v(x_i)\cdot f(x_i, u_i) \leq \epsilon,\\
		&l + \mathrm{grad}\:v\cdot f  \in Q_k(X \times U), \\
		&v =0 \:\:\mathrm{on}\:\:X_T,\\
		&\int_{\tilde{X} \times U} l + \mathrm{grad}\:v\cdot f = 1.
\end{array}
\end{equation}
where $l\in\RR[x,u]_{2k}$, $v\in\RR[x]_{2k}$, $\epsilon$ is a real, $\lambda>0$ (fixed) is a given regularization parameter, and $||.||_1$ denotes the $\ell_1$ norm of a polynomial, \textit{i.e.} the sum of absolute values of its coefficients when expanded in the monomial basis. The first constraints come from Definition \ref{def:approxSolution} and the last affine constraint is meant to avoid the trivial solution; see Proposition \ref{prop:normalization}. The $\ell_1$ norm is not differentiable around sparse vectors (with entries equal to zero) and  has the sparsity promoting role to bias solutions of the problem towards polynomial Lagrangian solutions with  few nonzero coefficients. 

\begin{center}
This regularization affects problem well-posedness and will prove to be essential in numerical experiments.
\end{center}

\subsubsection{Numerical implementation}
Linear constraints are easily expressed in term of polynomial coefficients. A classical lifting allows to express the $\ell_1$ norm as a linear program: for $x \in \RR^n$, $||x||_1 = \min_s\sum_{i=1}^n s_i$ subject to $s_i \geq x_i$ and $s_i \geq -x_i$, for all $i$. The Putinar positivity certificates can be expressed as LMIs \cite{lasserre2010} whose size depends on the degree bound $k$. We use the SOS module  of the \texttt{YALMIP} toolbox \cite{lofberg2009pre} to manipulate and express polynomial constraints at a high level in \texttt{MATLAB}. The size of the corresponding LMI grows as ${n + k \choose k}$ where $n$ is the number of variables and $k$ the degree bound in Putinar certificates. Thus it is reasonable to consider relatively small problems. As shown in the numerical results section, we could handle problems with 5 variables and degree 10 with a reasonable amount of time and memory. To handle larger size problems, specific heuristics and techniques beyond the scope of this paper 
must be implemented.

\subsubsection{General setting}
We consider several direct problems of the same form as (\ref{eq:pdirect0}). That is, we give ourselves compact basic semi-algebraic sets $X$, $U$, $X_T$, the dynamics $f$, and a Lagrangian $l_0$. We take known examples for which the (direct) optimal control law can be computed and try to vary their degree of difficulty. Given these optimal state-control trajectories, we generate randomly $n$ data points $\mathcal{D} = \{(x_i, u_i)\}_{i=1 \ldots n}$ according to the random process described in Section \ref{def:approxSolution}. For a given value of $\lambda$ and $k$, we compute a solution ${l}$ of problem (\ref{eq:primal1}). Then we measure how $l$ is close to $l_0$ by computing the following quantity (in the monomial basis):
\begin{align}
	\min_\alpha \frac{||l_0 - \alpha {l}||_2}{||l_0||_2}=\left(1 - \frac{\left\langle l_0, {l} \right\rangle^2}{||l_0||_2^2||{l}||_2^2}\right)^\frac{1}{2}.
	\label{eq:metric}
\end{align}
We also report  the value of $\epsilon$ in program (\ref{eq:primal1}). A larger value of $\epsilon$ means less reliable numerical certificates; see Remark \ref{rem:aproxSolution}. 

\subsection{Illustration on a one-dimensional example}
\label{sec:pb1}

\begin{figure}[t]
	\centering
	\includegraphics[width=.9\textwidth]{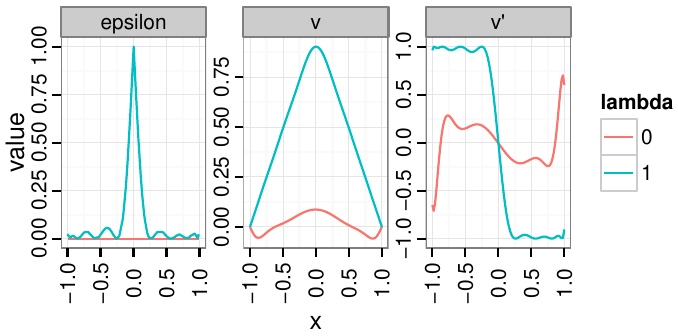}
	\caption{Solution for the one dimensional minimum exit time problem, effect of the regularization parameter $\lambda$. See Section \ref{sec:pb1} for problem details and comments. The first column is the distribution of the error $\epsilon$. It represents the value of $l(x, u) + \mathrm{grad}\:v\cdot f(x, u)$ as a function of $x$ where $u = \mathrm{sign}\:x$ is the optimal control. The second column is a representation of the value function ${v}$ and the third column is a representation of its derivative for solutions of problem (\ref{eq:primal1}) with and without regularization. We take 100 points on the segment. Lagrangian ${l}$ and  value function ${v}$ are both polynomials of degree 16.}
	\label{fig:oneDexit}
\end{figure}

First consider the eikonal problem of minimum exit time from the unit ball in the one-dimensional case. The data of the problem are
\begin{align*}
	X  = U =B_1, \, X_T = \partial X, \, l_0= 1, \, f = u.	
\end{align*}
The optimal law for this problem is $u = \text{sign}\:x$ and the value function is $v_0(x) = 1 - |x|$. We sample 100 points uniformly in $X$ and solve problem (\ref{eq:primal1}). We compare the choices $\lambda = 0$ (no regularization) and $\lambda = 1$. Results are presented in Figure \ref{fig:oneDexit} which displays the distribution of the error $\epsilon$, the estimated value function $v$ as well as its first derivative. Despite the simplicity of the problem, it is quite representative of the difficulties that arise in the context of inverse optimality. The first difficulty is the size of the set of solutions to the inverse problem:
\begin{itemize}
	\item given any symmetric differentiable concave function ${v}$ vanishing on $\{-1, 1\}$, the pair $({l} = |{v}'|, {v})$ solves problem (\ref{eq:primal1}) with $\epsilon = 0$;
	\item any positive polynomial on $X \times U$ vanishing if $|u| = 1$ solves problem (\ref{eq:primal1}) with $\epsilon = 0$;
	\item any Lagrangian of the form $v'(x) u$ solves problem (\ref{eq:primal1}) with $\epsilon = 0$;
	\item any convex combination of solutions of the types mentioned above also solves problem (\ref{eq:primal1}).
\end{itemize}
Even though formally accurate, these solutions form a relatively large set and do not carry any physical meaning. In the absence of any additional form of prior knowledge, it is impossible to discriminate between these solutions and the one that we wish to recover,
namely $l_0$. This is illustrated in Figure \ref{fig:oneDexit} where the red line ($\lambda = 0$) displays an example of value function $v$ obtained with a very low value of $\epsilon$. This is a very good certificate that our database $\mathcal{D}$ is close to optimal for the corresponding Lagrangian $l$. However, the estimated Lagrangian is far from the original one, namely $l_0 = 1$. Moreover, the shape of the value function is quite uncommon. This motivates the use of prior knowledge to bias the solutions of problem (\ref{eq:primal1}) toward a certain set of solutions. We use the $\ell_1$ norm which tends to promote Lagrangians with few non-zero coefficients.

When $\lambda = 1$, the sparsity inducing effect of $\ell_1$-norm regularization allows to recover the true Lagrangian ($l_0 = 1$) which is indeed sparse. The solution of problem (\ref{eq:primal1}) involves a polynomial function $v$ which should in principle be close to the true value function $v_0(x) = 1 - |x|$. The $v$ function displayed in Figure $1$ is close to $v_0$. However, $v_0$ is not smooth around the origin and therefore its derivative is harder to approximate by polynomials around this point. Hence the value of the error is higher around the origin.

\subsection{Illustration on more complex problems}
\begin{figure}[t]
	\centering
	\includegraphics[width=0.9\textwidth]{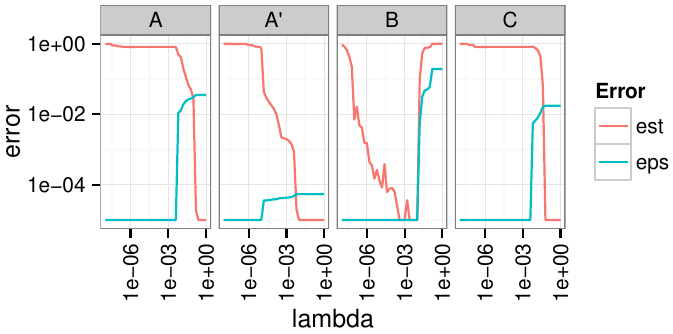}
	\caption{Error versus regularization parameter $\lambda$. Problems details are given in the main text. Estimation error (est) is given in (\ref{eq:metric}). Epsilon error (eps) is the value of $\epsilon$ in program (\ref{eq:primal1}). Trajectory sample size: 20 for A, A' and B, 50 for C. Degrees of ${l}$ and $v$ are 4 and 10 respectively.}
	\label{fig:deterministic}
\end{figure}
These simulations are taken from \cite{pauwels2014inverse}. We consider the following free terminal time direct problems:
\paragraph{Minimum exit time in dimension 2:}
\begin{align}
\label{sec:pb2}
	X = U = B_2, \, X_T = \partial X, \, l_0= 1, \, f = u.	
\end{align}
The optimal law is $u = \frac{x}{||x||_2}$ and the value function is $v_0(x) = 1 - ||x||_2$.
\paragraph{Minimum exit norm in dimension 2:}
\begin{align}
	\label{sec:pb3}
	X = U = B_2, \, X_T =  \partial X, l_0= ||x||_2^2 + ||u||_2^2, \, f = u.	
\end{align}
The optimal law is $u = x$ and the value function is $v_0(x) = 1 - ||x||_2^2$.
\paragraph{Minimum time Brockett integrator:}
\begin{align}
\label{sec:pb4}
	X &= 3B_3, \, U = B_2,\, l_0=1,\,f= (u_1, u_2, u_1x_2 - u_2x_1).
\end{align}
Recall that the Brockett integrator of nonlinear systems control
is also known (up to a change of coordinates) 
as the unicycle or Dubins system, one of the simplest instance of a non-holonomic system in robotics,
see e.g. \cite{devon2007} for the connection.
The optimal law and value function are described in \cite{prieur2005robust}. Complementary details are found in Appendix B of \cite{pauwels2014linear}.

\paragraph{Data generation:}
We consider the following settings:
\begin{itemize}
	\item[A] problem (\ref{sec:pb2}) with samples from $B_2$;
	\item[A'] problem (\ref{sec:pb2}) with samples from $B_2 \setminus \frac{1}{2}B_2$;
	\item[B] problem (\ref{sec:pb3}) with samples from $B_2$;
	\item[C] problem (\ref{sec:pb4}) with samples from $B_3 \setminus \left( \frac{1}{2}B_2 \times \RR \right)$.
\end{itemize}
In all cases we fix the degree of ${l}$ to $4$ and that of ${v}$ to $10$.

\paragraph{Results:} The results for the four problems are presented in Figure \ref{fig:deterministic}. For all problems, ${l}$ is of degree 4. Therefore, $l$ is to be found in a space of dimension 70 for problems A, A' and 126 for problem $D$. When the estimation error is close to 1, we estimate a Lagrangian ${l}$ that is orthogonal to $l_0$ (in the monomial basis), and when it is close to $0$, they are colinear. We also display the $\epsilon$ value of (\ref{eq:primal1}). We consider that the estimation is reasonable, when both estimation error and $\epsilon$ values are low.

$\bullet$ For all problems we are able to recover the true Lagrangian with good accuracy for some value of the regularization parameter $\lambda$. In the absence of regularization, we do not recover the true Lagrangian at all. This highlights the important role of $\ell_1$ regularization which allows to bias the estimation toward {\it sparse polynomials}. The choice of $\lambda$ in practical settings is subject to heuristics: numerical simulations or cross-validation which consists in keeping a portion of the input data as a validation set. 

$\bullet$ For all four problems, when the estimation error is minimal, the value of $\epsilon$ is reasonably low, depending on how the value function can be approximated by a polynomial. For example, A' shows lower $\epsilon$ value because we avoid sampling database points close to the non-differentiable point of the true value function. In example D, the value function is known to be harder to approximate by polynomials and the value of $\epsilon$ is a bit larger. The estimation accuracy is still very reasonable.

\section{Conclusion}

The main contribution of this paper is to propose a general framework to analyze the inverse problem of optimal control. The analysis is based on the weak formulation of direct optimal control problems using occupation measures, relaxed Hamilton-Jacobi-Bellman optimality conditions, and duality of infinite-dimensional linear programs. The proposed formulation is powerful enough to ensure that there is no gap between solutions of the direct and inverse problem (Theorem \ref{th:solution}). To the best of our knowledge this is the first result of this kind. In addition, in principle the proposed methodology is applicable to practical problems where we only have access to sample trajectories. We have also proposed numerical and statistical approximation procedures from which solid theoretical guaranties can be obtained. Finally we have illustrated our results on 
relatively simple (but not trivial) numerical examples of modest size.

One of the most striking aspects of the inverse problem is its set of valid solutions. Indeed, even for the simplest problems it is difficult to discriminate between physically meaningful Lagrangians and spurious mathematical solutions. For this reason, formulating the inverse problem as a well-posed problem (in particular with a unique solution) requires the introduction of strong prior knowledge -- sometimes arguably too restrictive -- about the nature of the Lagrangian to be recovered; see for example \cite{ajami2013humans}. However the proposed formulation based on relaxed HJB-optimality conditions, allows to get intuitions about characteristics that affect well-posedness of the problem. 

This work is to be seen as a first step toward a theoretical and practical framework for the resolution of inverse problems in a variety of contexts. Further aspects of the problem have to be investigated within this realm. First, we only deal with deterministic trajectories. For practical purposes it is essential to consider the effect of experimental noise, both from theoretical and practical perspectives, and to determine to which extent and how the problem can be solved in this more difficult context. Second, we have proposed a numerical scheme to approximate solutions and show that it is effective on academic examples of modest size. Experimental validation of such approximations should be carried out on real world examples of larger size. This involves a lot of data processing and fine tuning for each specific example. In this perspective, humanoid robotics provides an active and attractive field of application \cite{arechavaleta2008optimality,mombaur2010human}.

\section*{Acknowledgments}

This work was partly funded by an award of the {\it Simone and Cino del Duca foundation} of Institut de France,
a grant of the Gaspard Monge program (PGMO) of the  {\it Fondation math\'ematique Jacques Hadamard}. Most of this work was carried out during Edouard Pauwels' postdoctoral stay at LAAS-CNRS. The authors would like to thank Fr\'ed\'eric Jean, Jean-Paul Laumond, Nicolas Mansard and Ulysse Serres for fruitful discussions.

\appendix
\section{Proof of Proposition \ref{prop:unifApprox}}
\label{sec:statLearn}
We develop uniform finite sample bounds that hold with high probability for arbitrary probability distribution in the context of polynomial functions. These are in particular useful to derive bounds for the random process described by occupation measures as exposed in Section \ref{sec:discrete}. The techniques used have become fairly standard in empirical process theory and statistical learning theory, see for example \cite{bousquet2004introduction} for a nice introduction.

In what follows, we consider a compact set $Z \subset \RR^p$ with non empty interior. For a polynomial $z \in \RR^p \mapsto q \in \RR_d[z]$ of degree $d$, $c(q)$ denotes its coefficients in the monomial basis (of size $\binom{p+d}{d}$). Similarly for a point $z \in Z$, $v(z)$ denotes the $\binom{p+d}{d}$ dimensional vector representing the evaluation of the corresponding monomials at $z$ such that $q(z) = c(q)\:\cdot\:v(z)$ with the dot denoting the inner product. We consider the following set of polynomials $K_d(Z) = \{q \in \RR_d[z],\, q \geq 0 \:\:\text{on}\:\:Z,\, \int_{\tilde{Z}} q = 1 \}$ where $\tilde{Z}$ is a closed subset of $Z$ with nonempty interior. We fix an arbitrary probability distribution $P$ on $Z$. We denote by $\ell$ the linear functional on the space $\mathcal{C}(Z)$ such that 
$$\left\langle \ell, f \right\rangle = \int_Z f(z) P(dx).$$
Similarly for a sample of size $n$, $S_n = \{z_1, \ldots, z_n\} \in Z^n$, drawn \emph{iid} from $P$, we denote by $\ell_n$ the linear functional on the space $\mathcal{C}(Z)$ such that
$$\left\langle \ell_n, f \right\rangle = \frac{1}{n}\sum_{i=1}^n f(z_i)$$
for any function $f$ continuous on $Z$.
\begin{lemma}
	For any $0 \leq \delta \leq 1$ and any $q \in K_d(Z)$, it holds with probability $1 - \delta$ that
$$\left\langle \ell, q \right\rangle \leq \left\langle \ell_n, q \right\rangle + 2 \frac{M_c^d(Z) M_v^d(Z)}{\sqrt{n}} + M_\infty^d(Z) \sqrt{\frac{1}{2n} \ln \frac{2}{\delta}},$$
where 
\begin{align*}
	M_\infty^d(Z) &= \sup_{p \in K_d(Z),\, z \in Z} p(z),\\
	M_c^d(Z) &= \sup_{p \in K_d(Z)} ||c(p)||,\\
	M_v^d(Z) &= \sup_{z \in Z} ||v(z)||_2
\end{align*}
	are finite quantities that only depend on $X$ and $d$.
	\label{eq:finiteSampleBound}
\end{lemma}

\begin{proof}
	The proof combines standard arguments from statistical learning which we describe here for completeness.
In the sequel, given a probability distribution $P$ on $Z$, we use the notation
\[
\PP[A] = \int_A P(dz_1) \ldots P(dz_n), \quad
\EE_z[F] = \int_A F(z_1,\ldots,z_n) P(dz_1) \ldots P(dz_n)
\] 
 and we rely on the following concentration result:
	\begin{lemma}[McDiarmid's inequality \cite{mcdiarmid1989bounded}]
	Assume for all i = 1, \ldots , n,
	$$\sup_{z_1 ,\dots,z_n ,z_i' \in Z} |F (z_1 , \ldots , z_i , \ldots, z_n ) - F (z_1 , \ldots, z_i' , \ldots , z_n )| \leq \alpha,$$
	then, for all $\epsilon \geq 0$, when $\{z_i\}_{i=1}^n$ is drawn iid from a probability distribution $P$ on $Z$, we have
	$$\PP\left[ |F - \EE_z\left[ F \right]| \geq \epsilon \right] \leq 2 \exp\left( - \frac{2 \epsilon^2}{n\alpha^2} \right),$$
	where the expectation is taken over the random sample. An equivalent formulation is that for $0 \leq \delta \leq 1$, with probability $1 - \delta$, it holds that
	$$F \leq \EE_z\left[ F \right] + \alpha \sqrt{\frac{n}{2} \ln \frac{2}{\delta}}.$$
	\label{eq:thMcDiarmid}
\end{lemma}
	We consider the following quantity
	$$F(z_1, \ldots, z_n) := \sup_{p \in K_d(Z)} \left\langle \ell - \ell_n, p \right\rangle.$$
	Observe that $K_d(Z)$ is a subset of a finite dimensional space and that for $p \in K_d(Z)$, we have $||p||^{\tilde{Z}}_1 = \int_{\tilde{Z}} |p| = 1$. Since all norms are equivalent, $K_d(Z)$ is bounded in any given norm on polynomials, in particular, the supremum norm. Therefore, the quantity
	$$M_\infty^d(Z) :=  \sup_{p \in K_d(Z)} ||p||_\infty^{\tilde{Z}} := \sup_{p \in K_d(Z), z \in Z} p(z)$$
	is finite. We have that for all $i$, and any $z_1, \ldots, z_n, z_i'$ and any $p \in K_d(Z)$
	$$|F(z_1, \ldots, z_i, \ldots, z_n) - F(z_1, \ldots, z_i', \ldots, z_n)| \leq \sup_{p \in K_d{(Z)}} \frac{1}{n} |p(z_i) - p(z_i')| \leq \frac{M_\infty^d(Z)}{n}.$$
	Therefore McDiarmid's inequality of Lemma \ref{eq:thMcDiarmid} applies to function $F$ with $\alpha = \frac{M_\infty^d(X)}{n}$, and, for any $q \in K_d{(Z)}$, with probability $1 - \delta$, it holds that
	\begin{align}
		\left\langle \ell - \ell_n, q \right\rangle & \leq \sup_{p \in K_d(Z)} \left\langle \ell - \ell_n, p \right\rangle \leq \EE_z\left[\sup_{p \in K_d(Z)} \left\langle \ell - \ell_n, p \right\rangle\right] + M_\infty^d(Z) \sqrt{\frac{1}{2n} \ln \frac{2}{\delta}}. 
		\label{eq:intermediate1}
	\end{align}
The left hand side depends on the random draw of the sample $\{z_i\}_{i=1}^n$, but the right hand side is deterministic. We use a standard symmetrization argument to bound the expectation in the right hand side. Using the definition of $\ell$ and $\ell_n$, the convexity of the supremum and Jensen's inequality, we have that
	\begin{align}
		\label{eq:intermediate2}
		\EE_z\left[\sup_{p \in K_d(Z)} \left\langle \ell - \ell_n, p \right\rangle\right] &= \EE_z\left[\sup_{p \in K_d(Z)} \EE_{z'}\left[\left\langle \ell_n',p\right\rangle\right] - \left\langle \ell_n, p \right\rangle\right]\\ 
		&\leq \EE_{z,z'}\left[\sup_{p \in K_d(Z)} \left\langle \ell_n' - \ell_n,p\right\rangle\right] \nonumber\\
		&= \EE_{z,z'}\left[\sup_{p \in K_d(Z)} \frac{1}{n} \sum_{i=1}^n p(z_i') - p(z_i)\right]\nonumber
	\end{align}
	where the notation $z'$ refers to any other sample $S' = \{z_i'\}_{i=1}^n$ drawn from $P$ and $\ell_{n'}$ is the corresponding empirical measure. The \emph{iid}	assumption allows to flip $z_i$ and $z_i'$ in the expectation. Let $\xi_i$ be Rademacher variables, i.e. random variables which take values in $\{-1, 1\}$, each with probability one half. We have
	\begin{align}
		\label{eq:intermediate3}
		&\EE_{z,z'}\left[\sup_{p \in K_d(Z)} \frac{1}{n} \sum_{i=1}^n p(z_i') - p(z_i)\right] \\ 
		=&\EE_{z,z',\xi}\left[\sup_{p \in K_d(Z)} \frac{1}{n} \sum_{i=1}^n \xi_i(p(z_i') - p(z_i))\right]\nonumber\\
		\leq &\EE_{z,z',\xi}\left[\sup_{p \in K_d(Z)} \frac{1}{n} \sum_{i=1}^n \xi_i p(z_i') + \sup_{p \in K_d(Z)} \frac{1}{n} \sum_{i=1}^n -\xi_i p(z_i)\right] \nonumber\\
		= &2 \EE_{z,\xi}\left[\sup_{p \in K_d(Z)} \frac{1}{n} \sum_{i=1}^n \xi_i p(z_i)\right] \nonumber \\
		= &2 \EE_{z,\xi}\left[\sup_{p \in K_d(Z)} \frac{1}{n} c(p) \cdot \sum_{i=1}^n \xi_i v(z_i)\right]. \nonumber
	\end{align}
	The quantity on the right hand side is known as the Rademacher complexity of the function class $K_d(Z)$. Intuitively, it measures to which extent elements of a function class correlate with random noise in a worst case scenario. The function $p \to ||c(p)||_2$ is a norm on polynomials and since $K_d(Z)$ is bounded, the quantity
$$M_c^d(Z) := \sup_{p \in K_d(Z)} ||c(p)||_2$$
is finite. Moreover, since $Z$ is compact, the quantity
$$M_v^d(Z) := \sup_{z \in Z} ||v(z)||_2$$
is also finite and attained. We have that
\begin{align*}	
	\sup_{p \in K_d(Z)} \frac{1}{n} c(p) \cdot \sum_{i=1}^n \xi_i v(z_i) \leq &\frac{M_c^d(Z)}{n} \left|\left|\sum_{i=1}^n \xi_i v(z_i) \right|\right|_2 \\
	= & \frac{M_c^d(Z)}{n} \sqrt{\sum_{i=1}^n \sum_{j=1}^n \xi_i \xi_j v(z_i) \cdot v(z_j)}.
\end{align*}
Moreover, $\EE_\xi \left[\sum_{j=1}^n \xi_i \xi_j v(z_i) \cdot v(z_j)\right] = \sum_{i=1}^n v(z_i)^2$ ($\EE_\xi [\xi_i \xi_j] = I(i = j)$). Therefore, using Jensen's inequality (with concavity of the square root), we obtain
$$\EE_{\xi}\left[\sup_{p \in \mathcal{K}} \frac{1}{n} c(p) \cdot \sum_{i=1}^n \xi_i v(x_i)\right] \leq \frac{M_c^d(X)}{n} \sqrt{\sum_{i=1}^n ||v(x_i)||_2^2} \leq \frac{M_c^d(X) M_v^d(X)}{\sqrt{n}}.$$
Putting things together, using inequalities (\ref{eq:intermediate1}), (\ref{eq:intermediate2}), (\ref{eq:intermediate3}), we have that with probability $1 - \delta$, it holds
$$\left\langle c, q \right\rangle \leq \left\langle c_n, q \right\rangle + 2 \frac{M_c^d(X) M_v^d(X)}{\sqrt{n}} + M_\infty^d(X) \sqrt{\frac{1}{2n} \ln \frac{2}{\delta}}.$$

\end{proof}


\begin{thebibliography}{10}

\bibitem{abbeel2004apprenticeship}
P.~Abbeel and A. Y. Ng.
\newblock Apprenticeship learning via inverse reinforcement learning.
\newblock Proceedings of the International Conference on Machine Learning, ACM, 2004.

\bibitem{ajami2013humans}
A.~Ajami, J. P. Gauthier, T.~Maillot and U.~Serres.
\newblock How humans fly.
\newblock ESAIM: Control, Optimisation and Calculus of Variations,
19(4):1030--1054, 2013.

\bibitem{anderson1971linear}
B. D. O. Anderson and J.B. Moore.
\newblock Linear optimal control.
\newblock Prentice-Hall, Englewood Cliffs, NJ, 1971.

\bibitem{arechavaleta2008optimality}
G.~Arechavaleta, J. P. Laumond, H.~Hicheur and A.~Berthoz.
\newblock An optimality principle governing human walking.
\newblock IEEE Transactions on Robotics, 24(1):5--14, 2008.

\bibitem{athans1966optimal}
M.~Athans and P. L. Falb.
\newblock Optimal control. An introduction to the theory and its applications.
\newblock McGraw-Hill,  New York,1966.

\bibitem{bardi2008optimal}
M.~Bardi and I.~Capuzzo-Dolcetta.
\newblock Optimal control and viscosity solutions of Hamilton-Jacobi-Bellman equations.
\newblock Springer, Berlin, 2008.

\bibitem{alexander2002course}
A.~Barvinok.
\newblock A course in convexity.
\newblock AMS, Providence, NJ, 2002.

\bibitem{beals2000hamilton}
R.~Beals, B.~Gaveau, and P.~C. Greiner.
\newblock Hamilton-Jacobi theory and the heat kernel on Heisenberg groups.
\newblock Journal de Math{\'e}matiques Pures et Appliqu{\'e}es, 79(7):633--689, 2000.

\bibitem{bousquet2004introduction}
O.~Bousquet, S.~Boucheron, and G.~Lugosi.
\newblock Introduction to statistical learning theory.
\newblock In Advanced Lectures on Machine Learning, 169--207, Springer, Berlin, 2004.

\bibitem{casti1980general}
J.~Casti.
\newblock On the general inverse problem of optimal control theory.
\newblock Journal of Optimization Theory and Applications,
  32(4):491--497, 1980.

\bibitem{chittaro2013inverse}
F. C. Chittaro, F.~Jean, and P.~Mason.
\newblock On inverse optimal control problems of human locomotion: stability
  and robustness of the minimizers.
\newblock Journal of Mathematical Sciences, 195(3):269--287, 2013.

\bibitem{devon2007}
D. DeVon and T. Bretl. 
\newblock Kinematic and dynamic control of a wheeled mobile robot
\newblock IEEE/RSJ International Conference on Intelligent Robots and Systems, 2007.

\bibitem{frankowska2000filippov}
H.~Frankowska and F.~Rampazzo.
\newblock Filippov's and Filippov-Wazewski's theorems on closed domains.
\newblock Journal of Differential Equations. 161:449--478, 2000.

\bibitem{freeman1996inverse}
R.A. Freeman and P.V. Kokotovi\'c.
\newblock Inverse optimality in robust stabilization.
\newblock SIAM Journal on Control and Optimization, 34(4):1365--1391,
  1996.

\bibitem{friston2011what}
K.~Friston.
\newblock What is optimal about motor control? 
\newblock Neuron, 72(3):488--498, 2011.

\bibitem{fujii1984complete}
T.~Fujii and M.~Narazaki.
\newblock A complete optimality condition in the inverse problem of optimal
  control.
\newblock SIAM Journal on Control and Optimization, 22(2):327--341, 1984.

\bibitem{gaitsgory2009linear}
V.~Gaitsgory and M.~Quincampoix. 
\newblock Linear Programming Approach to Deterministic Infinite Horizon Optimal Control Problems with Discounting. 
\newblock SIAM J Control Optim 48(4):2480-2512, 2009.

\bibitem{henrion2014optim}
D.~Henrion. 
\newblock Optimization on linear matrix inequalities for polynomial systems control. 
\newblock Lecture notes of the International Summer School of Automatic Control, Grenoble, France, September 2014

\bibitem{henrion2014convex}
D.~Henrion and M.~ Korda, M.
\newblock Convex computation of the region of attraction of polynomial control systems.
\newblock IEEE Transactions on Automatic Control, 59(2):297--312, 2014.

\bibitem{henrion2014linear}
D.~Henrion and E.~Pauwels.
\newblock	Linear conic optimization for nonlinear optimal control.
\newblock	arXiv preprint arXiv:1407.1650, 2014

\bibitem{hernandez1996linear}
D.~Hern{\'a}ndez-Hern{\'a}ndez, O.~Hern{\'a}ndez-Lerma and M.~Taksar.
\newblock The linear programming approach to deterministic optimal control
  problems.
\newblock Applicationes Mathematicae, 24(1):17--33, 1996.

\bibitem{jameson1973inverse}
A.~Jameson and E.~Kreindler.
\newblock Inverse problem of linear optimal control.
\newblock SIAM Journal on Control, 11(1):1--19, 1973.

\bibitem{kalman1964linear}
R. E. Kalman.
\newblock When is a linear control system optimal?
\newblock Journal of Basic Engineering, 86(1):51--60, 1964.

\bibitem{kamien1991dynamic}
M.~Kamien and N. Schwartz. 
\newblock Dynamic optimization: the calculus fo variations and optimal control in economics and management.
\newblock Elsevier (1991).

\bibitem{keshavarz2011imputing}
A.~Keshavarz, Y.~Wang, and S. P. Boyd.
\newblock Imputing a convex objective function.
\newblock International Symposium on Intelligent Control, IEEE, 2011.

\bibitem{lasserre2010}
J. B. Lasserre.
\newblock Moments, positive polynomials and their applications.
\newblock Imperial College Press, UK, 2010.

\bibitem{lasserre2008nonlinear}
J.B. Lasserre, D.~Henrion, C.~Prieur, and E.~Tr{\'e}lat.
\newblock Nonlinear optimal control via occupation measures and LMI relaxations.
\newblock SIAM Journal on Control and Optimization, 47(4):1643--1666, 2008.

\bibitem{laumond2014optimality}
J.P.~Laumond, N.~Mansard and J.B.~Lasserre. 
\newblock Optimality in robot motion: optimal versus optimized motion. 
\newblock Communications of the ACM, 57(9):82--89, 2014.

\bibitem{lofberg2009pre}
J.~L\"ofberg.
\newblock Pre-and post-processing sum-of-squares programs in practice.
\newblock IEEE Transactions on Automatic Control, 54(5):1007--1011,
  2009.

\bibitem{mcdiarmid1989bounded}
C.~McDiarmid.
\newblock On the method of bounded differences.
\newblock In Surveys in Combinatorics, 148--188. Cambridge
  University Press, 1989.

\bibitem{mombaur2010human}
K.~Mombaur, A.~Truong, and J. P. Laumond.
\newblock From human to humanoid locomotion--an inverse optimal control
  approach.
\newblock Autonomous Robots, 28(3):369--383, 2010.

\bibitem{moylan1973nonlinear}
P.~Moylan and B. D. O. Anderson.
\newblock Nonlinear regulator theory and an inverse optimal control problem.
\newblock IEEE Transactions on Automatic Control, 18(5):460--465, 1973.

\bibitem{nori2004linear}
F.~Nori and R.~Frezza.
\newblock Linear optimal control problems and quadratic cost functions
  estimation.
\newblock Mediterranean Conference on Control and Automation, 2004.

\bibitem{pauwels2014inverse}
E.~Pauwels, D.~Henrion, and J.B. Lasserre.
\newblock Inverse optimal control with polynomial optimization.
\newblock IEEE Conference on Decision and Control, 2014.


\bibitem{pauwels2014linear}
E.~Pauwels, D.~Henrion, and J.B. Lasserre.
\newblock Linear conic optimization for inverse optimal control.
\newblock arXiv preprint arXiv:1412.2277, 2014.

\bibitem{prieur2005robust}
C.~Prieur and E.~Tr{\'e}lat.
\newblock Robust optimal stabilization of the Brockett integrator via a hybrid
  feedback.
\newblock Mathematics of Control, Signals and Systems, 17(3):201--216,
  2005.

\bibitem{putinar1993positive}
M.~Putinar.
\newblock Positive polynomials on compact semi-algebraic sets.
\newblock Indiana University Mathematics Journal, 42(3):969--984, 1993.

\bibitem{puydupin2012convex}
A.S. Puydupin-Jamin, M.~Johnson, and T.~Bretl.
\newblock A convex approach to inverse optimal control and its application to
  modeling human locomotion.
\newblock International Conference on  Robotics and Automation, IEEE, 2012.

\bibitem{ratliff2006maximum}
N.D. Ratliff, J.A. Bagnell, and M.A. Zinkevich.
\newblock Maximum margin planning.
\newblock International Conference on Machine Learning, ACM, 2006.

\bibitem{rosen1967optimality}
R.~Rosen, 
\newblock Optimality principles in biology. Springer 1967.

\bibitem{thau1967inverse}
F.~Thau.
\newblock On the inverse optimum control problem for a class of nonlinear
  autonomous systems.
\newblock IEEE Transactions on Automatic Control, 12(6):674--681, 1967.

\bibitem{todorov2004optimality}
E.~Todorov
\newblock Optimality principles in sensorimotor control. 
\newblock Nature neuroscience, 7(9):907-915, 2004.

\bibitem{vandenberghe1996semidefinite}
L.~Vandenberghe and S. P.~Boyd.
\newblock Semidefinite programming.
\newblock SIAM Review, 38(1):49--95, 1996.

\bibitem{vapnik1999overview}
V.N. Vapnik.
\newblock An overview of statistical learning theory.
\newblock IEEE Transactions on Neural Networks, 10(5):988--999, 1999.

\bibitem{vinter1978equivalence}
R.~Vinter and R.~Lewis. 
\newblock The equivalence of strong and weak formulations for certain problems in optimal control. 
\newblock SIAM J Control Optim 16(4):546-570, 1978. 

\bibitem{vinter1993convex}
R.~Vinter.
\newblock Convex duality and nonlinear optimal control.
\newblock SIAM Journal on Control and Optimization, 31(2):518--538, 1993.

\bibitem{whitney1957geometric}
H.~Whitney.
\newblock Geometric integration theory.
\newblock Princeton Univ. Press, 1957.

\end{thebibliography}
\end{document}